\documentclass[12pt]{amsart}
\usepackage{amscd,amssymb,amsthm,amsmath,amssymb,mathrsfs,enumerate}
\usepackage[matrix,arrow,curve]{xy}
\usepackage[margin=1.8cm]{geometry}

\newtheorem{theorem}{Theorem}
\newtheorem{proposition}[theorem]{Proposition}
\newtheorem*{proposition*}{Proposition}
\newtheorem{lemma}[theorem]{Lemma}
\newtheorem*{lemma*}{Lemma}
\newtheorem{corollary}[theorem]{Corollary}

\newtheorem*{theoremA*}{Theorem A}
\newtheorem*{theoremB*}{Theorem B}
\newtheorem*{auxiliarytheorem*}{Auxiliary Theorem}
\newtheorem*{corollary*}{Corollary}

\theoremstyle{definition}
\newtheorem{example}[theorem]{Example}
\newtheorem*{example*}{Example}

\theoremstyle{remark}
\newtheorem{remark}[theorem]{Remark}

\makeatletter\@addtoreset{equation}{section} \makeatother

\newcommand{\mumu}{\boldsymbol{\mu}}

\author{Ivan Cheltsov, Oliver Li, Sione Ma'u, Antoine Pinardin}

\title{K-stability and space sextic curves of genus three}

\thanks{Throughout this paper, all varieties are assumed to be projective and defined over~$\mathbb{C}$.}

\begin{document}

\begin{abstract}
We study Fano threefolds that can be obtained by blowing up the~three-dimensional projective space
along a~smooth curve of degree six and genus three.
We produce many new K-stable examples of such threefolds,
and we describe all finite groups that can act faithfully on them.
\end{abstract}

\address{\emph{Ivan Cheltsov}
\newline
\textnormal{University of Edinburgh,  Edinburgh, Scotland}
\newline
\textnormal{\texttt{I.Cheltsov@ed.ac.uk}}}

\address{\emph{Sione Ma'u}
\newline
\textnormal{University of Auckland, Auckland, New Zealand}
\newline
\textnormal{\texttt{s.mau@auckland.ac.nz}}}

\address{\emph{Oliver Li}
\newline
\textnormal{University of Melbourne, Melbourne, Australia}
\newline
\textnormal{\texttt{oli@unimelb.edu.au}}}

\address{\emph{Antoine Pinardin}
\newline
\textnormal{University of Edinburgh,  Edinburgh, Scotland}
\newline
\textnormal{\texttt{antoine.pinardin@ed.ac.uk}}}

\maketitle

\tableofcontents

\section{Introduction}
\label{section:intro}

Let $C$ be a~smooth quartic curve in $\mathbb{P}^2$, let $D$ be a~divisor of degree $2$ on the~curve $C$ such that
\begin{equation}
\label{equation:very-ample}\tag{$\diamondsuit$}
h^0\big(\mathcal{O}_{C}(D)\big)=0.
\end{equation}
Then $K_C+D$ is very ample \cite{Homma}, and the~linear system $|K_C+D|$ gives an embedding $\phi\colon C\hookrightarrow\mathbb{P}^3$.
We set $C_6=\phi(C)$. Then $C_6$ is a~smooth curve of degree $6$ and genus $3$.

Let $\pi\colon X\to\mathbb{P}^3$ be the~blow up of the~curve $C_6$.
Then $X$ is a~Fano threefold in the~deformation family \textnumero 2.12 in the~Mori--Mukai list, and every smooth member of this family can be obtained in this way.
Moreover, the~Fano threefold $X$ can be given in $\mathbb{P}^3\times\mathbb{P}^3$ by
\begin{equation}
\label{equation:P3-P3}\tag{$\clubsuit$}
(x_0,x_1,x_2,x_3)M_1\begin{pmatrix}
y_0\\
y_1\\
y_2\\
y_3
\end{pmatrix}=(x_0,x_1,x_2,x_3)M_2\begin{pmatrix}
y_0\\
y_1\\
y_2\\
y_3
\end{pmatrix}=(x_0,x_1,x_2,x_3)M_3\begin{pmatrix}
y_0\\
y_1\\
y_2\\
y_3
\end{pmatrix}=0
\end{equation}
for appropriate $4\times 4$ matrices $M_1$, $M_2$, $M_3$
such that $\pi$ is induced by the~projection to the~first factor,
where $([x_0:x_1:x_2:x_3],[y_0:y_1:y_2:y_3])$ are coordinates on $\mathbb{P}^3\times\mathbb{P}^3$.

Let $\pi^\prime\colon X\to\mathbb{P}^3$ be the~morphism induced by the~projection $\mathbb{P}^3\times\mathbb{P}^3\to\mathbb{P}^3$ to the~second factor.
Then $\pi^\prime$ is a~blow up of $\mathbb{P}^3$ along a~smooth curve $C_6^\prime$ of degree $6$ and genus $3$,
and the~$\pi^\prime$-exceptional surface is spanned by the~strict transforms of the~trisecants of the~curve $C_6$.
Furthermore, we have the~following commutative diagram:
\begin{equation}
\label{equation:diagram}\tag{$\bigstar$}
\xymatrix{
&X\ar@{->}[dl]_{\pi}\ar@{->}[dr]^{\pi^\prime}\\%
\mathbb{P}^3\ar@{-->}[rr]^\chi && \mathbb{P}^3}
\end{equation}
where $\chi$ is the~birational map given by the~linear system consisting of all cubic surfaces containing $C_6$.
Note that the~curves $C_6$ and $C_6^\prime$ are isomorphic, but they are not necessarily projectively isomorphic.

We can find the~equations of the~curves $C_6$ and $C_6^\prime$ as follows.
Rewrite \eqref{equation:P3-P3} as
$$
\begin{cases}
L_{10}y_0+L_{11}y_1+L_{12}y_2+L_{13}y_3=0,\\
L_{20}y_0+L_{21}y_1+L_{22}y_2+L_{23}y_3=0,\\
L_{30}y_0+L_{31}y_1+L_{32}y_2+L_{33}y_3=0,
\end{cases},
$$
where the~$L_{ij}$'s are linear functions in $x_0$, $x_1$, $x_2$, $x_3$. Set
$$
M=\begin{pmatrix}
L_{10} & L_{11} & L_{12} & L_{13}\\
L_{20} & L_{21} & L_{22} & L_{23}\\
L_{30} & L_{31} & L_{32} & L_{33}
\end{pmatrix}.
$$
Let $f_0$, $f_1$, $f_2$, $f_3$ be the~determinants of the~$3\times 3$ matrices obtained from the~matrix $M$ by removing its~first, second, third, fourth columns, respectively.
Then $C_6=\{f_0=0,f_1=0,f_2=0,f_3=0\}$, and the~birational map $\chi\colon\mathbb{P}^3\dasharrow\mathbb{P}^3$ in the~diagram \eqref{equation:diagram} is given by
$$
[x_0:x_1:x_2:x_3]\mapsto\big[f_0:f_1:f_2:f_3\big]
$$
up to a~composition with an automorphism of the~projective space $\mathbb{P}^3$.
Similarly, one can also describe the~defining equations of the~sextic curve~$C_6^\prime$.

\begin{example}[{\cite{Edge,Book}}]
\label{example:Edge}
Let
$$
X=\Big\{x_0y_1+x_1y_0-\sqrt{2}x_2y_2=0,x_0y_2+x_2y_0-\sqrt{2}x_3y_3=0,x_0y_3+x_3y_0-\sqrt{2}x_1y_1=0\Big\}\subset \mathbb{P}^3\times\mathbb{P}^3.
$$
Then $X$ is a~smooth Fano threefold in the~family \textnumero 2.12, the~curve $C_6$ is given by
$$
\left\{\aligned
&2\sqrt{2}x_1x_2x_3-x_0^3=0,\\
&x_0^2x_1+\sqrt{2}x_0x_2^2+2x_2x_3^2=0,\\
&x_0^2x_2+\sqrt{2}x_0x_3^2+2x_1^2x_3=0,\\
&x_0^2x_3+\sqrt{2}x_0x_1^2+2x_1x_2^2=0,
\endaligned
\right.
$$
and $C_6^\prime$ is given by the~same equations replacing each $x_i$ by $y_i$.
One has $\mathrm{Aut}(X)\simeq\mathrm{PSL}_2(\mathbb{F}_7)\times\mumu_2$,
and $X$ is the~only smooth Fano threefold in the~deformation family \textnumero 2.12 that admits a~faithful action of the~Klein simple group $\mathrm{PSL}_2(\mathbb{F}_7)$.
The map $\chi$ in \eqref{equation:diagram} can be chosen to be an involution.
\end{example}

The following result has been proven in \cite{Book}.

\begin{theorem}[{\cite[\S~5.4]{Book}}]
\label{theorem:Klein-group}
Let $X$ be the~Fano threefold from Example~\ref{example:Edge}. Then $X$ is K-stable.
\end{theorem}

Hence, a~general member of the~family \textnumero 2.12 is K-stable, since K-stability is an open condition.
We expect that every smooth Fano threefold in this family is K-stable.
To~show this, it is enough to prove that
$$
\beta(\mathbf{F})=A_X(\mathbf{F})-S_X(\mathbf{F})>0
$$
for every prime divisor $\mathbf{F}$ over $X$ \cite{Fujita2019,Li}, where $A_X(\mathbf{F})$ is the~log discrepancy of the~divisor $\mathbf{F}$, and
$$
S_X\big(\mathbf{F}\big)=\frac{1}{(-K_X)^3}\int\limits_0^{\infty}\mathrm{vol}\big(-K_X-u\mathbf{F}\big)du.
$$
Unfortunately, we are unable to prove this result at the~moment. Instead, we prove a~weaker result.
To state it, let $E$ be the~$\pi$-exceptional surface, and let $E^\prime$ be the~$\pi^\prime$-exceptional surface.

\begin{theoremA*}
Let $\mathbf{F}$~be~a~prime divisor over $X$ such that $\beta(\mathbf{F})\leqslant 0$, and let $Z$ be its center on $X$.
Then $Z$ is a~point in the~intersection $E\cap E^\prime$.
\end{theoremA*}

Let us present applications of this result.
By \cite[Corollary~4.14]{Zhuang}, Theorem A implies

\begin{corollary}
\label{corollary:1}
If $\mathrm{Aut}(X)$ does not fix points in $E\cap E^\prime$, then $X$ is K-stable.
\end{corollary}

Since the~action of the~group $\mathrm{Aut}(\mathbb{P}^3,C_6)$ lifts to $X$, Corollary~\ref{corollary:1} implies

\begin{corollary}
\label{corollary:2}
If $\mathrm{Aut}(\mathbb{P}^3,C_6)$ does not fix a~point in $C_6$, then $X$ is $K$-stable.
\end{corollary}

Since the group $\mathrm{Aut}(\mathbb{P}^3,C_6)$ acts faithfully on the~curve $C_6$, Corollary~\ref{corollary:2} implies
the following generalization of Theorem~\ref{theorem:Klein-group},
which has more applications (see Section~\ref{section:examples}).

\begin{corollary}
\label{corollary:3}
If $\mathrm{Aut}(\mathbb{P}^3,C_6)$ is not cyclic, then $X$ is $K$-stable.
\end{corollary}

\begin{proof}
If the~group $\mathrm{Aut}(\mathbb{P}^3,C_6)$ fixes a~point $P\in C_6$, it acts faithfully on the~one-dimensional tangent space to the~curve $C_6$ at the~point $P$ by \cite[Lemma~2.7]{FZ},
so that $\mathrm{Aut}(\mathbb{P}^3,C_6)$ is cyclic.
\end{proof}

What do we know about $\mathrm{Aut}(X)$? This group is finite \cite{CPS},
and we have the~following~exact sequence:
$$
1\rightarrow \mathrm{Aut}\big(\mathbb{P}^3,C_6\big) \rightarrow \mathrm{Aut}(X) \rightarrow \mumu_2,
$$
where $\mathrm{Aut}(\mathbb{P}^3,C_6)\simeq\mathrm{Aut}(C,[D])$, and the~final homomorphism is surjective $\iff$
$\mathrm{Aut}(X)$ contains an~element that swaps $E$ and $E^\prime$.
For~instance, if $X$ is the~smooth Fano threefold from Example~\ref{example:Edge},
then the~group $\mathrm{Aut}(X)$ contains such an~element --- it is the~involution given~by
$$
\big([x_0:x_1:x_2:x_3],[y_0:y_1:y_2:y_3]\big)\mapsto\big([y_0:y_1:y_2:y_3],[x_0:x_1:x_2:x_3]\big),
$$
which implies that $\mathrm{Aut}(X)\simeq\mathrm{PSL}_2(\mathbb{F}_7)\times\mumu_2$ in this case.
In Section~\ref{section:Aut}, we will discuss the~possibilities for the~group $\mathrm{Aut}(X)$ in more details.
In particular, we will present a criterion when $\mathrm{Aut}(X)$ contains an~element that swaps $E$ and $E^\prime$,
and we will prove the~following result (cf. \cite[Theorem 1.1]{WeiYu}).

\begin{theoremB*}
A finite group $G$ has a~faithful action on a~smooth Fano threefold in the~deformation family \textnumero 2.12
if and only if $G$ is isomorphic to a~subgroup of $\mathrm{PSL}_2(\mathbb{F}_7)\times\mumu_2$ or $\mumu_4^2\rtimes\mathfrak{S}_3$.
\end{theoremB*}

As we mentioned in Example~\ref{example:Edge},
the~family \textnumero 2.12 contains a unique smooth Fano threefold that admits a~faithful action of the~group $\mathrm{PSL}_2(\mathbb{F}_7)$.
Similarly, we prove in Section~\ref{section:Aut} that the~deformation family \textnumero 2.12 contains a unique smooth threefold that admits a~faithful action of the~group
$\mumu_4^2\rtimes\mumu_3$, and the full automorphism group of this threefold is $\mumu_4^2\rtimes\mathfrak{S}_3$.

\begin{remark}
\label{remark:G-Fano}
Let $G$ be a subgroup in $\mathrm{Aut}(X)$.
If $G$ has an~element that swaps the surfaces $E$ and $E^\prime$, then $X$ is a $G$-Mori fiber space (over a point), and $X$ is also known as a $G$-Fano threefold (see~\cite{Prokhorov2013}).
In this case, it is natural to ask the following three nested questions:
\begin{enumerate}
\item Is there a $G$-equivariant birational map $X\dasharrow \mathbb{P}^3$? Cf. \cite{CheltsovTschinkelZhang,CiurcaTanimotoTschinkel,KuznetsovProkhorov}.
\item Is $X$ $G$-solid? Cf. \cite{CheltsovSarikyan,Pinardin}.
\item Is $X$ $G$-birationally rigid? Cf. \cite{CheltsovShramov}.
\end{enumerate}
Inspired by \cite[Corollary 6.11]{KuznetsovProkhorov}, we conjecture that the answer to the first question is always negative.
If~$G\simeq \mathrm{PSL}_2(\mathbb{F}_7)\times\mumu_2$, then $X$ is $G$-birationally rigid \cite[Theorem 5.23]{Book}, so, in particular, it is $G$-solid.
We believe that $X$ is also $G$-birationally rigid if $G\simeq\mumu_4^2\rtimes\mathfrak{S}_3$.
\end{remark}

To consider more applications of Theorem A, let $\Bbbk$ be a~subfield in $\mathbb{C}$ such that $C_6$ is defined~over~$\Bbbk$.
Then $X$ and the~Sarkisov link \eqref{equation:diagram} are defined over~$\Bbbk$.
In~particular, the~curve $C_6^\prime$ is defined over~$\Bbbk$.
Moreover, it follows from \cite{BenoistWittenberg,LauterSerre} that $C_6^\prime$ and $C_6$ are isomorphic over~$\Bbbk$,
which can be shown directly.
By \cite[Corollary~4.14]{Zhuang}, Theorem A implies the~following corollaries.

\begin{corollary}
\label{corollary:4}
If $E\cap E^\prime$ does not have $\Bbbk$-points, then $X$ is K-stable.
\end{corollary}

\begin{corollary}
\label{corollary:5}
If $C_6$  does not have $\Bbbk$-points, then $X$ is K-stable.
\end{corollary}

Using \cite[Corollary~4.14]{Zhuang}, we also obtain

\begin{corollary}
\label{corollary:7}
Every smooth Fano threefold in the~deformation family \textnumero 2.12 which is defined over a~subfield of the~field $\mathbb{C}$
and does not have points in this subfield is K-stable.
\end{corollary}

We will present applications of Corollaries~\ref{corollary:5} and Corollary~\ref{corollary:7} in Section~\ref{section:examples}.

\medskip
\noindent
\textbf{Acknowledgements.}
We started this project in Auckland (New Zealand) back in December 2022
during Ivan Cheltsov and Antoine Pinardin's visit, which has been supported by the~Leverhulme Trust grant RPG-2021-229. Oliver Li is supported by the~ Australian Commonwealth Government.

We would like to thank Harry Braden, Dougal Davis, Linden Disney-Hogg, Tim Dokchitser, Igor Dolgachev, Sasha Kuznetsov, Yuri Prokhorov, Giorgio Ottaviani, Costya Shramov and Yuyang Zhou for their help in this project.

\section{Examples}
\label{section:examples}

\subsection{$\mathfrak{S}_4$-invariant curves}
\label{subsection:S4}

Let us use notations introduced in Section~\ref{section:intro}.
Suppose, in addition, that
$$
M_1=\begin{pmatrix}0&a&1&0\\a&0&0&-1\\1&0&0&a\\0&-1&a&0\end{pmatrix},
M_2=\begin{pmatrix}0&1&0&a\\1&0&a&0\\0&a&0&-1\\a&0&-1&0\end{pmatrix},
M_2=\begin{pmatrix}0&0&a&-1\\0&0&1&a\\a&1&0&0\\-1&a&0&0\end{pmatrix},
$$
where $a\in\mathbb{C}$ such that $a(a^6-1)\ne 0$.
Then $X$ is a~smooth Fano threefold in the~family \textnumero 2.12, and
$$
M=\begin{pmatrix}
ax_1+x_2&ax_0-x_3&ax_3+x_0&ax_2-x_1\\
ax_3+x_1&ax_2+x_0&ax_1-x_3&ax_0-x_2\\
ax_2-x_3&ax_3+x_2&ax_0+x_1&ax_1-x_0
\end{pmatrix},
$$
so that $C_6=\{f_0=0,f_1=0,f_2=0,f_3=0\}$ for
\begin{multline*}
f_0=(1-a^3)x_0^3-(2a^2+2a)x_0^2x_1+(2a^2+2a)x_0^2x_2+\\
+(2a^2+2a)x_0^2x_3+(a^3-1)x_0x_1^2-(2a^2-2a)x_0x_1x_2-(2a^2-2a)x_0x_1x_3+\\
+(a^3-1)x_0x_2^2+(2a^2- 2a)x_0x_2x_3+(a^3-1)x_0x_3^2-(2a^3+2)x_1x_2x_3,
\end{multline*}
\begin{multline*}
f_1=(1-a^3)x_0^2x_1+(-2a^2-2a)x_0x_1^2+(2a^2-2a)x_0x_1x_2-\\
-(2a^2-2a)x_0x_1x_3+(2a^3+2)x_0x_2x_3+(a^3-1)x_1^3+(2a^2+2a)x_1^2x_2-\\
-(2a^2+2a)x_1^2x_3+(-a^3+ 1)x_1x_2^2+(2a^2-2a)x_1x_2x_3+(1-a^3)x_1x_3^2,
\end{multline*}
\begin{multline*}
f_2=(a^3-1)x_0^2x_2-(2a^2-2a)x_0x_1x_2-(2a^3+2)x_0x_1x_3+\\
+(-2a^2-2a)x_0x_2^2-(2a^2-2a)x_0x_2x_3+(a^3-1)x_1^2x_2+(2a^2+2a)x_1x_2^2+\\
+(2a^2-2a)x_1x_2x_3+ (1-a^3)x_2^3+(2a^2+2a)x_2^2x_3+(a^3-1)x_2x_3^2,
\end{multline*}
\begin{multline*}
f_3=(1-a^3)x_0^2x_3+(2a^3+2)x_0x_1x_2-(2a^2-2a)x_0x_1x_3-\\
-(2a^2-2a)x_0x_2x_3+(2a^2+2a)x_0x_3^2+(1-a^3)x_1^2x_3-(2a^2-2a)x_1x_2x_3+\\
+(2a^2+2a)x_1x_3^2+ (1-a^3)x_2^2x_3+(2a^2+2a)x_2x_3^2+(a^3-1)x_3^3.
\end{multline*}

It follows from \cite{Ottaviani} that the~curve $C_6$ is isomorphic to the~plane quartic curve in $\mathbb{P}^2_{x,y,z}$
that is given by the~equation $\det(xM_1+yM_2+zM_3)=0$, which can be rewritten as
$$
x^4+y^4+z^4+\lambda(x^2y^2+x^2z^2+y^2z^2)=0
$$
for $\lambda=-\frac{2a^4+2}{(a^2+1)^2}$, cf. \cite[\S~14]{Edge38}.
So, it follows from \cite{DolgachevBook} that $\mathrm{Aut}(C_6)\simeq\mathfrak{S}_4$ if $\lambda\ne 0$ and $\lambda^2+3\lambda+18\ne 0$.
Moreover, if $\lambda=0$, then $\mathrm{Aut}(C_6)\simeq\mumu_4^2\rtimes\mathfrak{S}_3$.
Furthermore, if $\lambda^2+3\lambda+18=0$, then $C_6$ is isomorphic to the~Klein quartic curve,
and $\mathrm{Aut}(C_6)\simeq\mathrm{PSL}_2(\mathbb{F}_7)$.

\begin{lemma}
\label{lemma:S4}
The group $\mathrm{Aut}(\mathbb{P}^3,C_6)$ contains a~subgroup isomorphic to $\mathfrak{S}_4$.
\end{lemma}

\begin{proof}
Let $G$ be the~subgroup in $\mathrm{PGL}_4(\mathbb{C})$ that is generated by the~following transformations:
$$
\begin{pmatrix}0&0&0&i\\0&0&i&0\\0&-i&0&0\\-i&0&0&0\end{pmatrix},
\begin{pmatrix}0&0&i&0\\0&0&0&-i\\-i&0&0&0\\0&i&0&0\end{pmatrix},
\begin{pmatrix}1&-3&-1&1\\-3&-1&1&1\\-1&1&-3&1\\1&1&1&3\end{pmatrix},
\begin{pmatrix}-3&-1&1&1\\-1&1&-3&1\\1&-3&-1&1\\1&1&1&3\end{pmatrix}.
$$
Then, using Magma, one can check that $G\simeq\mathfrak{S}_4$.
Moreover, the~curve $C_6$ is $G$-invariant.
\end{proof}

\begin{corollary}
\label{corollary:S4}
If $\lambda\ne 0$ and $\lambda^2+3\lambda+18\ne 0$, then $\mathrm{Aut}(\mathbb{P}^3,C_6)\simeq\mathfrak{S}_4$.
\end{corollary}

Similarly, we prove

\begin{lemma}
\label{lemma:S4-threefold}
The group $\mathrm{Aut}(X)$ contains a~subgroup isomorphic to $\mathfrak{S}_4\times\mumu_2$.
\end{lemma}

\begin{proof}
Let $G$ be the~subgroup in $\mathrm{PGL}_4(\mathbb{C})$ that is defined in the~proof of Lemma~\ref{lemma:S4}.
Then $G\simeq\mathfrak{S}_4$, the~group $G$ acts diagonally on $\mathbb{P}^3\times\mathbb{P}^3$, and $X$ is $G$-invariant.
This gives an embedding $\mathfrak{S}_4\hookrightarrow\mathrm{Aut}(X)$.
Moreover, since the~matrices $M_1$, $M_2$, $M_3$ are symmetric, the~involution
$$
\big([x_0:x_1:x_2:x_3],[y_0:y_1:y_2:y_3]\big)\mapsto\big([y_0:y_1:y_2:y_3],[x_0:x_1:x_2:x_3]\big)
$$
leaves $X$ invariant and commutes with the~$\mathfrak{S}_4$-action, which implies the~required assertion.
\end{proof}

\begin{corollary}
\label{corollary:S4-threefold}
If $\lambda\ne 0$ and $\lambda^2+3\lambda+18\ne 0$, then $\mathrm{Aut}(X)\simeq\mathfrak{S}_4\times\mumu_2$.
\end{corollary}

Applying Corollary~\ref{corollary:3}, we conclude that the~Fano threefold $X$ is K-stable.

\subsection{$\mumu_4^2\rtimes\mumu_3$-invariant curve}
\label{subsection:48-3}

Let $\widehat{G}$ be the~subgroup in $\mathrm{GL}_4(\mathbb{C})$ generated by the~matrices
$$
M=\begin{pmatrix}
-1 & 0 & 0 & 0\\
0 & 1 & 0 & 0\\
0 & 0 & -1 & 0\\
0 & 0 & 0 & 1
\end{pmatrix},
N=\begin{pmatrix}
1 & 0 & 0 & 0\\
0 & -1 & 0 & 0\\
0 & 0 & -1 & 0\\
0 & 0 & 0 & 1
\end{pmatrix},
A=\begin{pmatrix}
0 & 0 & 1 & 0\\
1 & 0 & 0 & 0\\
0 & 1 & 0 & 0\\
0 & 0 & 0 & 1
\end{pmatrix},
B=\begin{pmatrix}
0 & i & 0 & 0\\
1 & 0 & 0 & 0\\
0 & 0 & 0 & -i\\
0 & 0 & 1 & 0
\end{pmatrix}.
$$
and let $G$ be the~image of the~group $\widehat{G}$ in $\mathrm{PGL}_4(\mathbb{C})$ via the~natural projection $\mathrm{GL}_4(\mathbb{C})\to\mathrm{PGL}_4(\mathbb{C})$.
Then $\widehat{G}\simeq\mumu_4.(\mumu_4^2\rtimes\mumu_3)$ and $G\simeq\mumu_4^2\rtimes\mumu_3$, and their GAP ID's are [192,4] and [48,3], respectively.
Using GAP \cite{GAP}, one can check that $H^2(G,\mathbb{C}^\ast)\simeq\mumu_4$, and $\widehat{G}$ is a~covering group of the~group $G$.

\begin{lemma}
\label{lemma:48-3-unique}
Let $G^\prime$ be a~subgroup in $\mathrm{PGL}_4(\mathbb{C})$ such that $G^\prime\simeq G$ and $G^\prime$ does not fix points in $\mathbb{P}^3$.
Then $G^\prime$ is conjugate to $G$ in $\mathrm{PGL}_4(\mathbb{C})$.
\end{lemma}

\begin{proof}
The claim follows from \cite[Lemma 2.7]{CheltsovSarikyan}
and the~classification of finite subgroups~in~$\mathrm{PGL}_4(\mathbb{C})$, which can be found in \cite{Blichfeldt1917}.
Alternatively, one can prove the~required assertion analyzing irreducible representations of the~group $\widehat{G}$,
which can be found in \cite{Tim}.
\end{proof}

The main goal of this subsection is to show that the~projective space $\mathbb{P}^3$ contains
a $G$-invariant irreducible smooth non-hyperelliptic curve of degree $6$ and genus $3$,
and this curve is unique up to the~action of the~normalizer of the~group $G$ in $\mathrm{PGL}_4(\mathbb{C})$.
First, let us describe the~normalizer. Set
$$
 C_4^{\pm}=\big\{\big(1\mp\sqrt{3}i\big)x_1^2-(1\pm\sqrt{3}i)x_2^2+2x_3^2=0,2x_0^2-(1\pm\sqrt{3}i)x_1^2-\big(1\mp\sqrt{3}i\big)x_2^2=0\big\}\subset\mathbb{P}^3.
$$
Then $C_{4}^{\pm}$ is a~$G$-invariant elliptic curve,
and $\mathrm{Aut}(\mathbb{P}^3,C_{4}^{\pm})$ is the~subgroup in $\mathrm{PGL}_4(\mathbb{C})$ generated by
$$
\begin{pmatrix}
-1 & 0 & 0 & 0\\
0 & 1 & 0 & 0\\
0 & 0 & 1 & 0\\
0 & 0 & 0 & 1
\end{pmatrix},
\begin{pmatrix}
1 & 0 & 0 & 0\\
0 & -1 & 0 & 0\\
0 & 0 & 1 & 0\\
0 & 0 & 0 & 1
\end{pmatrix},
\begin{pmatrix}
1 & 0 & 0 & 0\\
0 & 1 & 0 & 0\\
0 & 0 & -1 & 0\\
0 & 0 & 0 & 1
\end{pmatrix},
\begin{pmatrix}
0 & 0 & 1 & 0\\
1 & 0 & 0 & 0\\
0 & 1 & 0 & 0\\
0 & 0 & 0 & 1
\end{pmatrix},
\begin{pmatrix}
0 & i & 0 & 0\\
1 & 0 & 0 & 0\\
0 & 0 & 0 & i\\
0 & 0 & 1 & 0
\end{pmatrix}.
$$
Note that $\mathrm{Aut}(\mathbb{P}^3,C_{4}^{\pm})\simeq\mumu_2^3.\mathfrak{A}_4$.
Let $G_{192,185}$ be the~subgroup in $\mathrm{PGL}_4(\mathbb{C})$ generated by
$$
\begin{pmatrix}
-1 & 0 & 0 & 0\\
0 & 1 & 0 & 0\\
0 & 0 & 1 & 0\\
0 & 0 & 0 & 1
\end{pmatrix},
\begin{pmatrix}
1 & 0 & 0 & 0\\
0 & -1 & 0 & 0\\
0 & 0 & 1 & 0\\
0 & 0 & 0 & 1
\end{pmatrix},
\begin{pmatrix}
1 & 0 & 0 & 0\\
0 & 1 & 0 & 0\\
0 & 0 & -1 & 0\\
0 & 0 & 0 & 1
\end{pmatrix},
\begin{pmatrix}
0 & 0 & 0 & 1\\
i & 0 & 0 & 0\\
0 & 1 & 0 & 0\\
0 & 0 & 1 & 0
\end{pmatrix},
\begin{pmatrix}
0 & i & 0 & 0\\
-i & 0 & 0 & 0\\
0 & 0 & i & 0\\
0 & 0 & 0 & 1
\end{pmatrix}.
$$
Then its GAP ID is [192,185]. Note that $\mathrm{Aut}(\mathbb{P}^3,C_{4}^{\pm})\vartriangleleft G_{192,185}\simeq\mumu_2^3.\mathfrak{S}_4$ and $G\vartriangleleft G_{192,185}$.

\begin{lemma}
\label{lemma:48-normalizer-PGL}
The normalizer in $\mathrm{PGL}_4(\mathbb{C})$ of the~subgroup $G$ is the~subgroup $G_{192,185}$.
\end{lemma}

\begin{proof}
This follows from the~fact that the~curve $C_{4}^++C_{4}^-$ is $G_{192,185}$-invariant.
\end{proof}

Let us describe $G$-orbits in $\mathbb{P}^3$ of length less than $48$. To do this, we let
\begin{align*}
\Sigma_{4}&=\mathrm{Orb}_G\big([1:0:0:0]\big),\\
\Sigma_{12}&=\mathrm{Orb}_G\big([1+i:\sqrt{2}:0:0]\big),\\
\Sigma_{12}^\prime&=\mathrm{Orb}_G\big([1-i:\sqrt{2}:0:0]\big),\\
\Sigma_{16}&=\mathrm{Orb}_G\big([-1+\sqrt{3}i:-1-\sqrt{3}i:2:0]\big),\\
\Sigma_{16}^\prime&=\mathrm{Orb}_G\big([-1-\sqrt{3}i:-1+\sqrt{3}i:2:0]\big),\\
\Sigma_{16}^u&=\mathrm{Orb}_G\big([1:1:1:u]\big)\ \text{for}\ u\in\mathbb{C},\\
\Sigma_{24}^t&=\mathrm{Orb}_G\big([2:t:0:0]\big)\ \text{for}\ t\in\mathbb{C}\ \text{such that}\ t\ne 0\ \text{and}\ t\ne\pm \sqrt{2}\pm\sqrt{2}i.
\end{align*}
Then $\Sigma_{4}$, $\Sigma_{12}$, $\Sigma_{12}^\prime$, $\Sigma_{16}$, $\Sigma_{16}^\prime$, $\Sigma_{16}^u$, $\Sigma_{24}^t$
are $G$-orbits of length $4$, $12$, $12$, $16$, $16$, $16$, $24$, respectively.

\begin{lemma}
\label{lemma:G-subgroups}
Let $\Sigma$ be a~$G$-orbit in $\mathbb{P}^3$ such that $|\Sigma|<48$.
Then $\Sigma$ is one of the~$G$-orbits
\begin{center}
$\Sigma_{4}$, $\Sigma_{12}$, $\Sigma_{12}^\prime$, $\Sigma_{16}$, $\Sigma_{16}^\prime$, $\Sigma_{16}^u$, $\Sigma_{24}^t$,
\end{center}
where $u\in\mathbb{C}$ and $t\in\mathbb{C}$ such that $0\ne t\ne\pm \sqrt{2}\pm\sqrt{2}i$.
\end{lemma}

\begin{proof}
Let us describe subgroups of the~group $G$.
To do this, identify the~matrices $M$, $N$, $A$, $B$ with their images in $\mathrm{PGL}_4(\mathbb{C})$.
Set
$$
C=ANBMA^2=\begin{pmatrix}
0 & 0 & 0 & 1\\
0 & 0 & 1 & 0\\
0 & i & 0 & 0\\
-i & 0 & 0 & 0
\end{pmatrix}.
$$
Then, using \cite{Tim}, we see that all proper subgroups of the~group $G$ can be described as follows:
\begin{itemize}
\item[(i)] $\langle B,C\rangle\simeq\mumu_4^2$ is the~unique (normal) subgroup of order $16$,
\item[(ii)] $\langle A,M,N\rangle\simeq\mathfrak{A}_4$ is one of four conjugated subgroups of order $12$,
\item[(iii)] $\langle B,M,N\rangle\simeq\mumu_2\times\mumu_4$ is one of three conjugated subgroups of order $8$,
\item[(iv)] $\langle M,N\rangle\simeq\mumu_2^2$ is the~unique (normal) subgroup isomorphic to $\mumu_2^2$,
\item[(v)] $\langle B\rangle\simeq\mumu_4$ and $\langle CB\rangle\simeq\mumu_4$ are non-conjugate subgroups,
their conjugacy classes consist of three subgroups, which are all subgroups of the~group $G$ isomorphic to~$\mumu_4$,
\item[(vi)] $\langle A\rangle\simeq\mumu_3$ is one of sixteen conjugated subgroups of order $3$,
\item[(vii)] $\langle M\rangle\simeq\mumu_2$ is one of three conjugated subgroups of order $2$.
\end{itemize}

Now, let $\Gamma$ be the~stabilizer in $G$ of a~point in $\Sigma$.
Then $\Gamma$ is a~proper subgroup of the~group $G$, since $G$ fixes no points in~$\mathbb{P}^3$.
So, we may assume that $\Gamma$ is one of the~subgroups $\langle B,C\rangle$, $\langle A,M,N\rangle$, $\langle B,M,N\rangle$, $\langle M,N\rangle$, $\langle B\rangle$, $\langle CB\rangle$, $\langle A\rangle$, $\langle M\rangle$.
On the~other hand, one can check that
\begin{itemize}
\item[(i)] $\langle B,C\rangle$ does not fix points in $\mathbb{P}^3$,
\item[(ii)] the~only fixed point of $\langle A,M,N\rangle$ is the~point $[1:0:0:0]\in\Sigma_{4}$,
\item[(iii)] $\langle B,M,N\rangle$ does not fix points in $\mathbb{P}^3$,
\item[(iv)] $\langle M,N\rangle$ does not fix points in $\mathbb{P}^3\setminus\Sigma_{4}$,
\item[(v)]  the~only fixed point of $\langle B\rangle$ are the~points
$$
[1+i:\sqrt{2}:0:0],[1+i:-\sqrt{2}:0:0],[0:0:\sqrt{2}:1+i],[0:0:-\sqrt{2}:1+i],
$$
which are contained in $\Sigma_{12}$, and the~only fixed point of $\langle CB\rangle$ are the~points
$$
[\sqrt{2}:0:1-i:0],[-\sqrt{2}:0:1-i:0],[0:\sqrt{2}:0:1-i],[0:-\sqrt{2}:0:1-i],
$$
which are contained in the~$G$-orbit $\Sigma_{12}^\prime$,
\item[(vi)] the~only fixed points of $\langle A\rangle$ are the~points
\begin{itemize}
\item $[-1+\sqrt{3}i:-1-\sqrt{3}i:2:0]\in\Sigma_{16}$,
\item $[-1-\sqrt{3}i:-1+\sqrt{3}i:2:0]\in\Sigma_{16}^\prime$,
\item $[1:1:1:t]\in\Sigma_{16}^t$ for any $t\in\mathbb{C}$,
\item $[0:0:0:1]\in\Sigma_4$,
\end{itemize}
\item[(vii)] all fixed points of $\langle M\rangle$ are contained in the~lines $\{x_0=x_1=0\}$ and $\{x_2=x_3=0\}$.
\end{itemize}
This implies the~required assertion.
\end{proof}

Now, we are ready to present a~$G$-invariant irreducible smooth curve in $\mathbb{P}^3$ of degree $6$ and genus~$3$.
For every $u\in\mathbb{C}$ such that $u\ne 0$, let $\mathcal{M}_3^u$ be the~linear subsystem in $|\mathcal{O}_{\mathbb{P}^3}(3)|$
that consists of all cubic surfaces passing through the~$G$-orbit $\Sigma_{16}^u$.
If $u^{4}=-3$, then the~linear system  $\mathcal{M}_3^u$ is $7$-dimensional, and its base locus
consists of one of the~two elliptic curves $C_4^+$ or $C_4^-$.
One the~other hand, if $u^4\ne -3$, then the~linear subsystem $\mathcal{M}_3^u$ is  $3$-dimensional,
and its base locus is given by
\begin{equation}
\label{equation:four-cubics}\tag{$\heartsuit$}
\left\{\aligned
&(u^4-1)x_3x_0^2+(u^4+3)x_0x_1x_2u+(u^4-1)x_3x_1^2-4x_3^3u^2+(u^4-1)x_2^2x_3=0,\\
&(u^4-1)x_1x_0^2-u(u^4+3)x_0x_2x_3+4u^2x_1^3-(u^4-1)x_1x_2^2+(u^4-1)x_3^2x_1=0,\\
&4u^2x_0^3-(u^4-1)x_0x_1^2+(u^4-1)x_0x_2^2+(u^4-1)x_3^2x_0-u(u^4+3)x_1x_2x_3=0,\\
&(u^4-1)x_2x_0^2+u(u^4+3)x_0x_1x_3-(u^4-1)x_2x_1^2-4u^2x_2^3-(u^4-1)x_3^2x_2=0.
\endaligned
\right.
\end{equation}
Using this, one can check that the~base locus is zero-dimensional unless
$$
u\in\Bigg\{\frac{-1\pm \sqrt{3}}{2}+\frac{1\mp \sqrt{3}}{2}i,\frac{-1\pm \sqrt{3}}{2}+\frac{-1\pm \sqrt{3}}{2}i,
\frac{1\pm \sqrt{3}}{2}+\frac{1\pm \sqrt{3}}{2}i,\frac{1\mp \sqrt{3}}{2}+\frac{-1\pm \sqrt{3}}{2}i\Bigg\}.
$$
On the~other hand, if $u=\frac{-1\pm \sqrt{3}}{2}+\frac{1\mp \sqrt{3}}{2}i$, $u=\frac{-1\pm \sqrt{3}}{2}+\frac{-1\pm \sqrt{3}}{2}i$,
$u=\frac{1\pm \sqrt{3}}{2}+\frac{1\pm \sqrt{3}}{2}i$ or $u=\frac{1\mp \sqrt{3}}{2}+\frac{-1\pm \sqrt{3}}{2}i$,
then the~equations \eqref{equation:four-cubics} define an irreducible $G$-invariant smooth curve in $\mathbb{P}^3$ of degree $6$~and~genus~$3$.
We will denote these curves by $C_6$, $C_6^\prime$,  $C_6^{\prime\prime}$, $C_6^{\prime\prime\prime}$, respectively.
To be precise, we have
$$
C_6=\left\{\aligned
&(x_0^2+x_1^2+x_2^2)x_3-ix_3^3-(1-i)x_1x_0x_2=0,\\
&(x_0^2-x_2^2+x_3^2)x_1+ix_1^3+(1-i)x_3x_0x_2=0,\\
&(x_1^2-x_2^2-x_3^2)x_0-ix_0^3-(1-i)x_1x_3x_2=0,\\
&(x_0^2-x_1^2-x_3^2)x_2-ix_2^3-(1-i)x_1x_3x_0=0,
\endaligned
\right.
$$
and $C_6^\prime$,  $C_6^{\prime\prime}$, $C_6^{\prime\prime\prime}$
can be obtained from $C_6$ by applying elements of the~normalizer $G_{192,185}$.

Fix $u=\frac{-1\pm \sqrt{3}}{2}+\frac{1\mp \sqrt{3}}{2}i$.
Then \eqref{equation:four-cubics} defines $C_6$.
Choosing a~different basis of the~linear system $\mathcal{M}_3^u$,
we obtain a~birational map $\iota\colon\mathbb{P}^3\dasharrow\mathbb{P}^3$ given by $[x_0:x_1:x_2:x_3]\mapsto[h_0:h_1:h_2:h_3]$ for
\begin{align*}
h_0&=(1+i)x_3x_0x_2-x_1^3+ix_1(x_0^2-x_2^2+x_3^2),\\
h_1&=(1+i)x_3x_1x_2-x_0^3-ix_0(x_1^2-x_2^2-x_3^2),\\
h_2&=(1+i)x_3x_0x_1-x_2^3-ix_2(x_0^2-x_1^2-x_3^2),\\
h_3&=(i-1)x_1x_0x_2-ix_3^3+x_3(x_0^2+x_1^2+x_2^2).
\end{align*}
One can check that $\iota$ is a~birational involution,
and we have the~following $G$-commutative diagram:
$$
\xymatrix{
X\ar@{->}[d]_{\pi}\ar@{->}[rr]^{\tau}&&X\ar@{->}[d]^{\pi}\\%
\mathbb{P}^3\ar@{-->}[rr]_{\iota}&&\mathbb{P}^3}
$$
where $\pi$ is the~blow up of the~curve  $C_6$, and $\tau$ is an involution.
Then $X$ is smooth Fano threefold in the~deformation family \textnumero 2.12,
which can be defined as complete intersection in $\mathbb{P}^3\times\mathbb{P}^3$ given by
$$
\left\{\aligned
&y_3x_0-y_2x_0+iy_2x_1+y_3x_1-y_0x_2+iy_1x_2+y_0x_3+y_1x_3=0,\\
&iy_0x_0-y_1x_1+y_3x_2+y_2x_3=0,\\
&y_2x_0+y_3x_0+iy_2x_1-y_3x_1-y_0x_2-iy_1x_2-y_0x_3+y_1x_3=0,
\endaligned
\right.
$$
and $\pi$ is induced by the~projection to the~first factor,
where $([x_0:x_1:x_2:x_3],[y_0:y_1:y_2:y_3])$ are coordinates on $\mathbb{P}^3\times\mathbb{P}^3$.
Thus, in the~notations in Section~\ref{section:intro}, we have
$$
M_1=
\begin{pmatrix}
0&0&-1&1\\
0&0&i&1\\
-1&i&0&0\\
1&1&0&0
\end{pmatrix},
M_2=\begin{pmatrix}
i&0&0&0\\
0&-1&0&0\\
0&0&0&1\\
0&0&1&0
\end{pmatrix},
M_3=
\begin{pmatrix}
0&0&1&1\\
0&0&i&-1\\
-1&-i&0&0\\
-1&1&0&0
\end{pmatrix}.
$$
Note that $M_1$ and $M_2$ are symmetric, $M_3$ is skew-symmetric,
and the~involution $\tau$ is given by
$$
\big([x_0:x_1:x_2:x_3],[y_0:y_1:y_2:y_3]\big)\mapsto\big([y_0:y_1:y_2:y_3],[x_0:x_1:x_2:x_3]\big),
$$

\begin{corollary}
\label{corollary:Fermat}
One has $\mathrm{Aut}(\mathbb{P}^3,C_6)=G$ and $\mathrm{Aut}(X)\simeq\mumu_4^2\rtimes\mathfrak{S}_3$.
\end{corollary}

\begin{proof}
First, using the~classification of automorphism groups of smooth curves of genus three \cite{DolgachevBook,Bars},
we see that $C_6$ is isomorphic to the~Fermat quartic curve in $\mathbb{P}^2$.
This can also be shown directly.
Namely, it follows from \cite{Ottaviani} that $C_6$ is isomorphic to the~plane quartic curve
$$
\big\{\mathrm{det}(xM_1+yM_2+zM_3)=0\big\}\subset\mathbb{P}^2_{x,y,z},
$$
which is projectively isomorphic to the~Fermat plane quartic curve.

We conclude that $\mathrm{Aut}(C_6)\simeq\mumu_4^2\rtimes\mathfrak{S}_3$.
Therefore, if $\mathrm{Aut}(\mathbb{P}^3,C_6)\ne G$, then $\mathrm{Aut}(\mathbb{P}^3,C_6)\simeq \mumu_4^2\rtimes\mathfrak{S}_3$,
and the~subgroup $\mathrm{Aut}(\mathbb{P}^3,C_6)\subset\mathrm{PGL}_4(\mathbb{C})$ is contained in the~normalizer of the~group $G$ in $\mathrm{PGL}_4(\mathbb{C})$,
which is impossible since the~normalizer is the~group $G_{192,185}$ by Lemma~\ref{lemma:48-normalizer-PGL},
and $G_{192,185}$ does not contain subgroups isomorphic to $\mumu_4^2\rtimes\mathfrak{S}_3$.

Therefore, we conclude that $\mathrm{Aut}(\mathbb{P}^3,C_6)=G$.
Now, one can explicitly check that $\langle G,\tau\rangle\simeq\mumu_4^2\rtimes\mathfrak{S}_3$,
where we consider $G$ as a~subgroup in $\mathrm{Aut}(X)$.
This gives $\mathrm{Aut}(X)\simeq\mumu_4^2\rtimes\mathfrak{S}_3$.
\end{proof}

By Corollary~\ref{corollary:3}, the~smooth Fano threefold $X$ is K-stable.

In Section~~\ref{section:Aut}, we will see that $X$
is the~unique smooth Fano threefold in the~family \textnumero 2.12 whose automorphism group is isomorphic to the~group $\mumu_4^2\rtimes\mathfrak{S}_3$.
To do this, we need the~following result:

\begin{theorem}
\label{theorem:d-6-g-3}
The only $G$-invariant irreducible smooth curves in $\mathbb{P}^3$ of degree~$6$ are $C_6$, $C_6^\prime$,  $C_6^{\prime\prime}$, $C_6^{\prime\prime\prime}$.
\end{theorem}

\begin{proof}
Let $C$ be a~$G$-invariant irreducible smooth curve in $\mathbb{P}^3$ of degree $6$,
and let $g$ be its genus.
Then $g\leqslant 4$ by the~Castelnuovo bound.
Thus, it follows from \cite{Breuer,Paulhus} that either $g=1$, or $g=3$.

Note that $\Sigma_4\not\subset C$, because stabilizers in $G$ of points in $C$ are cyclic by \cite[Lemma~2.7]{FZ}.

Let $\Pi=\{x_3=0\}$,
and let $\Gamma$ be the~the stabilizer of this plane in $G$.
Then $\Gamma=\langle M,N,A\rangle\simeq\mathfrak{A}_{4}$,
and all $\Gamma$-orbits in $\Pi$ of length less than $12$ can be described as follows:
\begin{enumerate}[\normalfont(i)]
\item $\Sigma_{4}\cap\Pi$ is the~unique $\Gamma$-orbit of length $3$,
\item $\Sigma_{12}\cap\Pi$ is a~$\Gamma$-orbit of length $6$,
\item $\Sigma_{12}^\prime\cap\Pi$ is a~$\Gamma$-orbit of length $6$,
\item $\Sigma_{24}^t\cap\Pi$ is a~$\Gamma$-orbit of length $6$, where $0\ne t\ne\pm \sqrt{2}\pm\sqrt{2}i$,
\item $\Sigma_{16}\cap\Pi$, $\Sigma_{16}^\prime\cap\Pi$ and $\Sigma_{16}^0\cap\Pi$ are $\Gamma$-orbits of length $4$.
\end{enumerate}
Thus, since $\Sigma_4\not\subset C$, $C\not\subset\Pi$, and $\Pi\cdot C$ is a~$\Gamma$-invariant effective one-cycle of degree $6$,
we conclude that $C$ contains at least one of the~orbits $\Sigma_{12}$ or $\Sigma_{12}^\prime$,
and $C$ does not contain $\Sigma_{16}$, $\Sigma_{16}^\prime$ and $\Sigma_{16}^0$.

If $g=1$, it follows from \cite{Breuer,Paulhus} that $C$ does not contain $G$-orbits of length $12$, which gives $g=3$.
Then it follows from \cite{Breuer,Paulhus} that $C$ contains two $G$-orbits of length $16$,
so $\Sigma_{16}^t\subset C$ for some $t\ne 0$.

Using the~classification of automorphism groups of smooth curves of genus three \cite{DolgachevBook,Bars},
we see that the~curve $C$ is isomorphic to the~Fermat quartic curve in $\mathbb{P}^2$.
Hence, the~curve $C$ is not hyperelliptic.

Let $\mathcal{M}_3$ be the~linear subsystem in $|\mathcal{O}_{\mathbb{P}^3}(3)|$ that consists of all cubic surfaces passing through~$C$.
Then $\mathcal{M}_3$ is three-dimensional, and the~curve $C$ is its base locus by \cite{Homma}, because $C$ is not hyperelliptic.
Therefore, using the~notations introduced earlier, we see that $\mathcal{M}_3=\mathcal{M}_3^t$ for an appropriate $t\in\mathbb{C}$.
Now, arguing as above, we see that
$$
t\in\Bigg\{\frac{-1\pm \sqrt{3}}{2}+\frac{-1\pm \sqrt{3}}{2}i,\frac{-1\pm \sqrt{3}}{2}+\frac{1\mp \sqrt{3}}{2}i,
\frac{1\pm \sqrt{3}}{2}+\frac{1\pm \sqrt{3}}{2}i,\frac{1\mp \sqrt{3}}{2}+\frac{-1\pm \sqrt{3}}{2}i\Bigg\},
$$
which implies that $C$ is one of the~curves $C_6$, $C_6^\prime$,  $C_6^{\prime\prime}$, $C_6^{\prime\prime\prime}$ as claimed.
\end{proof}

\subsection{Curves over $\mathbb{Q}$ without rational points}
\label{subsection:Q}

Let us use notations introduced in Section~\ref{section:intro}.
Suppose, in addition, that
$$
C=\{x^{4}+xyz^{2}+y^{4}+y^{3}z-31yz^{3}+4z^{4}=0\}\subset\mathbb{P}^2_{x,y,z}.
$$
Then $C$ is smooth.
One can show that $C(\mathbb{Q})=\varnothing$ using the~reduction modulo $3$. Set
$$
P_1=[1-i:0:1], P_2=[1+i:0:1], P_3=[-1+i:0:1], P_4=[-1-i:0:1].
$$
and $D=3(P_3+P_4)-K_C$. Then $D$ is defined over $\mathbb{Q}$, and $D$ satisfies \eqref{equation:very-ample}.
Then $C_6$ is defined over~$\mathbb{Q}$, and it is isomorphic to $C$ over $\mathbb{Q}$.
In particular, the~curve $C_6$ does not contains $\mathbb{Q}$-rational points.
Hence, by Corollary~\ref{corollary:5}, the~smooth Fano threefold $X$ is K-stable.

One can explicitly find defining equations of $C_6$ as follows.
Let $\mathcal{M}$ be the~linear system of cubic curves in $\mathbb{P}^2$ whose general member is tangent to $C$
with multiplicity $3$ at the~points $P_1$ and $P_2$. Then
$$
\mathcal{M}\big\vert_{C}=3P_1+3P_2+|3(P_3+P_4)|.
$$
Thus, to compute the~embedding $C\hookrightarrow\mathbb{P}^3$, it is enough to find a~basis of the~linear system $\mathcal{M}$,
which can be done using linear algebra. After this, it is easy to find defining equations of the~curve $C_6$.

\subsection{Real pointless threefolds}
\label{subsection:pointless}

Now, we explain how to construct real smooth Fano threefolds in the~deformation family \textnumero 2.12 that do not have real points.
By Corollary~\ref{corollary:7}, all of them are K-stable. We start with

\begin{example}
\label{example:Severi-Brauer}
Let $U$ be a~three-dimensional Severi--Brauer variety defined over $\mathbb{R}$ such that $U\not\simeq\mathbb{P}^3_{\mathbb{R}}$.
Recall~from \cite{GilleSzamuely,Kollar} that $U$ exists, it is unique, and, in particular,
it is isomorphic to its dual variety.
Set $W=U\times U$. Then
$$
W_{\mathbb{C}}\simeq\mathbb{P}^3\times\mathbb{P}^3.
$$
Since $U\simeq U^\vee$, the~Picard group $\mathrm{Pic}_{\mathbb{R}}(W)$ contains a~real line bundle $L$ such that $L_{\mathbb{C}}$ has degree~$(1,1)$.
Let $V$ be any smooth complete intersection of three divisors in $|L|$.
Then $V$ is  a~smooth Fano threefold in the~family \textnumero 2.12,
and $V$ does not have real points, because $W$ does not have real points.
\end{example}

Let us present another, more explicit, construction of pointless real smooth Fano threefolds in the~deformation family \textnumero 2.12.
For a point $P=([x_0:x_1:x_2:x_3],[y_0:y_1:y_2:y_3])\in\mathbb{P}^3\times\mathbb{P}^3$, let us consider the~symmetric matrix
$$
A=\begin{pmatrix}
a_{00} & a_{01} & a_{02} & a_{03} &  \\
a_{01} & a_{11} & a_{12} & a_{13} &  \\
a_{02} & a_{12} & a_{22} & a_{23} &  \\
a_{03} & a_{13} & a_{23} & a_{33} &
\end{pmatrix}
$$
and the~skew-symmetric matrix
$$
B=\begin{pmatrix}
0 & b_{01} & b_{02} & b_{03} &  \\
-b_{01} & 0 & b_{12} & b_{13} &  \\
-b_{02} & -b_{12} & 0 & b_{23} &  \\
-b_{03} & -b_{13} & -b_{23} & 0 &
\end{pmatrix}
$$
defined (up to a common scalar multiple) as follows:
$$
a_{nm}=\frac{x_ny_m+x_my_n}{2}
$$
and
$$
b_{nm}=\frac{x_ny_m-x_my_n}{2i}
$$
for every  $n\in\{0,1,2,3\}$ and $m\in\{0,1,2,3\}$ such that $n\ne m$, and $a_{nn}=x_ny_n$ for each~$n\in\{0,1,2,3\}$.
Set $M=A+iB$. Then
$$
M=\begin{pmatrix}
x_0y_0 & x_0y_1 & x_0y_2 & x_0y_3\\
x_1y_0 & x_1y_1 & x_1y_2 & x_1y_3\\
x_2y_0 & x_2y_1 & x_2y_2 & x_2y_3\\
x_3y_0 & x_3y_1 & x_3y_2 & x_3y_3.
\end{pmatrix}
$$
Therefore, we see that the~constructed map $P\mapsto M$ gives us the~Serge embedding $\mathbb{P}^3\times\mathbb{P}^3\hookrightarrow\mathbb{P}^{15}$,
where we consider $\mathbb{P}^{15}$ as a projectivization of the~vector space of all $4\times 4$ matrices.

Now, we consider matrices $A$ and $B$ on their own, and we also assume that all $a_{ij}$ and $b_{ij}$ are real.
Then $M$ is a Hermitian $4\times 4$ matrix.
Projectivizing the~vector space of Hermitian $4\times 4$ matrices,
we obtain $\mathbb{P}^{15}_{\mathbb{R}}$ with coordinates $[a_{00}:a_{01}:\cdots:b_{13}:b_{23}]$.
Let us consider $M$ as a point in $\mathbb{P}^{15}_{\mathbb{R}}$, and set
$$
V=\big\{M\in \mathbb{P}^{15}_{\mathbb{R}}\ \big\vert\ \mathrm{rank}(M)\leqslant 1\big\}\subset \mathbb{P}^{15}_{\mathbb{R}}.
$$
Then $V$ is a real projective subvariety in $\mathbb{P}^{15}_{\mathbb{R}}$.
Moreover, over $\mathbb{C}$, the~subvariety $V_{\mathbb{C}}$ is the~image of the~map $P\mapsto M$ constructed above,
which implies that $V_{\mathbb{C}}\simeq\mathbb{P}^3\times\mathbb{P}^3$,
so $V$ is a form of $\mathbb{P}^3_{\mathbb{R}}\times\mathbb{P}^3_{\mathbb{R}}$.
But $V\not\simeq\mathbb{P}^3_{\mathbb{R}}\times\mathbb{P}^3_{\mathbb{R}}$ over $\mathbb{R}$,
because $V$ is the~Weil restriction of $\mathbb{P}^3$ over the~reals  \cite[Exercise 8.1.6]{GorchinskyShramov},
which implies that $V(\mathbb{R})\ne\varnothing$, and $\mathrm{Pic}_{\mathbb{R}}(V)$ is generated by the~class of a hyperplane section.

Now, let $H_1$, $H_2$, $H_3$ be three real hyperplane sections of $V\subset\mathbb{P}^{15}_{\mathbb{R}}$, and let $X=H_1\cap H_2\cap H_3$.
Suppose that $X$ is smooth and three-dimensional.
Then $X$ is a real form of a smooth Fano threefold in the~deformation family \textnumero 2.12 such that $\mathrm{Pic}_{\mathbb{R}}(X)=\mathbb{Z}[-K_X]$.
Moreover, Corollary~\ref{corollary:7} gives

\begin{corollary}
\label{corollary:7-real}
If $X$ does not have real points, then $X$ is K-stable
\end{corollary}

Such smooth Fano threefolds without real points do exists:

\begin{example}
\label{example:pointless}
Suppose that $H_1$ is cut out by $a_{00}+a_{11}+a_{22}+a_{33}=0$.
Then $H_1$ is smooth, because its preimage in $\mathbb{P}^3\times\mathbb{P}^3$ via the~map constructed above is given~by
$$
x_0y_0+x_1y_1+x_2y_2+x_3y_3=0.
$$
Moreover, the~fivefold $H_1$ does not have real points.
Indeed, if $M\in V$, then the~corresponding real numbers $a_{00}$, $a_{11}$, $a_{22}$, $a_{33}$ are either all non-negative or all non-positive,
and they cannot be all zero.
Similarly, set $H_2=\{a_{03}+2a_{12}=0\}\cap V$ and $H_3=\{a_{02}+a_{13}+a_{23}=0\}\cap V$.
Then $V_{\mathbb{C}}$ is isomorphic to the~complete intersection in $\mathbb{P}^3\times\mathbb{P}^3$ given by
$$
\left\{\aligned
&x_0y_0+x_1y_1+x_2y_2+x_3y_3=0,\\
&x_0y_3 + x_3y_0 +2x_1y_2 + 2x_2y_1=0,\\
&x_0y_2 + x_2y_0 + x_3y_1 + x_1y_3+x_2y_3+x_3y_2=0.
\endaligned
\right.
$$
This complete intersection is a smooth threefold, so $X$ is smooth, and it has no real points, because the~divisor $H_1$ does not have real points.
\end{example}

\section{The proof of Theorem A}
\label{section:proof}

Let us use all notations and assumptions introduced in Section~\ref{section:intro}.
To start with, let us present few results from \cite{AbbanZhuang,Book} that will be used in the~proof of Theorem~A.
Let $\mathbf{F}$ be a~prime divisor over~$X$,
and let $Z$ be its center on~$X$. Suppose that
\begin{itemize}
\item either $Z$ is a~point,
\item or $Z$ is an irreducible curve.
\end{itemize}
Let $P$ be any point in $Z$. Choose an~irreducible smooth surface $S\subset X$ such that $P\in S$. Set
$$
\tau=\mathrm{sup}\Big\{u\in\mathbb{Q}_{\geqslant 0}\ \big\vert\ \text{the divisor  $-K_X-uS$ is pseudo-effective}\Big\}.
$$
For~$u\in[0,\tau]$, let $P(u)$ be the~positive part of the~Zariski decomposition of the~divisor $-K_X-uS$,
and let $N(u)$ be its negative part. Then $\beta(S)=1-S_X(S)$, where
$$
S_X(S)=\frac{1}{-K_X^3}\int\limits_{0}^{\infty}\mathrm{vol}\big(-K_X-uS\big)du=\frac{1}{20}\int\limits_{0}^{\tau}P(u)^3du.
$$
Let us show how to compute $P(u)$ and $N(u)$. Set $H=\pi^*(\mathcal{O}_{\mathbb{P}^3}(1))$ and $H^\prime=(\pi^\prime)^*(\mathcal{O}_{\mathbb{P}^3}(1))$. Then
$$
H\sim 3H^\prime-E^\prime, E\sim 8H^\prime-3E^\prime, H^\prime\sim 3H-E, E^\prime\sim 8H-3E,
$$
where $E$ and $E^\prime$ are exceptional surfaces of the~blow ups $\pi$ and $\pi^\prime$, respectively.

\begin{example}
\label{example:AZ-surfaces-H}
Suppose that $S\in |H|$. Then $\tau=\frac{4}{3}$. Moreover, we have
$$
P(u)\sim_{\mathbb{R}} \left\{\aligned
&(4-u)H-E \ \text{ for } 0\leqslant u\leqslant 1, \\
&(4-3u)H^\prime\ \text{ for } 1\leqslant u\leqslant \frac{4}{3},
\endaligned
\right.
$$
and
$$
N(u)= \left\{\aligned
&0\ \text{ for } 0\leqslant u\leqslant 1, \\
&(u-1)E^\prime\ \text{ for } 1\leqslant u\leqslant \frac{4}{3},
\endaligned
\right.
$$
which gives
$S_{X}(S)=\frac{1}{20}\int\limits_{0}^{\frac{4}{3}}\big(P(u)\big)^3du=\frac{1}{20}\int\limits_0^{1}(2-u)(u^2-10u+10)du+\frac{1}{20}\int\limits_{1}^{\frac{4}{3}}(4-3u)^3du=\frac{53}{120}$.
\end{example}

\begin{example}
\label{example:AZ-surfaces-E}
Suppose that $S=E$. Then $\tau=\frac{1}{2}$,
$$
P(u)\sim_{\mathbb{R}} \left\{\aligned
&4H-(1+u)E \ \text{ for } 0\leqslant u\leqslant \frac{1}{3}, \\
&(4-8u)H^\prime\ \text{ for } \frac{1}{3}\leqslant u\leqslant \frac{1}{2},
\endaligned
\right.
$$
and
$$
N(u)=\left\{\aligned
&0\ \text{ for } 0\leqslant u\leqslant \frac{1}{3}, \\
&(3u-1)E^\prime\ \text{ for } \frac{1}{3}\leqslant u\leqslant \frac{1}{2},
\endaligned
\right.
$$
which gives
$S_{X}(S)=\frac{1}{20}\int\limits_0^{\frac{1}{2}}4(1-u)(5-7u^2-10u)du+\frac{1}{20}\int\limits_{\frac{1}{2}}^{\frac{1}{3}}64(1-2u)^3du=\frac{11}{60}$.
\end{example}

Now, we choose an~irreducible curve $C\subset S$ that contains the~point $P$.
For instance, if $Z$ is a~curve, and $S$ contains $Z$, then we can choose $C=Z$.
Since $S\not\subset\mathrm{Supp}(N(u))$,  we can write
$$
N(u)\big\vert_S=d(u)C+N^\prime(u),
$$
where $d(u)=\mathrm{ord}_C(N(u)\vert_S)$, and $N^\prime(u)$ is an~effective $\mathbb{R}$-divisor on $S$ such that $C\not\subset\mathrm{Supp}(N^\prime(u))$.
Now, for every $u\in [0,\tau]$, we set
$$
t(u)=\sup\Big\{v\in \mathbb R_{\geqslant 0} \ \big|\ \text{the divisor $P(u)\big|_S-vC$ is pseudo-effective}\Big\}.
$$
For $v\in [0, t(u)]$, we let $P(u,v)$ be the~positive part of the~Zariski decomposition of $P(u)|_S-vC$,
and we let $N(u,v)$ be its negative part. Following \cite{AbbanZhuang,Book}, we let
$$
 S\big(W^S_{\bullet,\bullet};C\big)=\frac{3}{(-K_X)^3}\int\limits_0^{\tau}d(u)\Big(P(u)\big\vert_{S}\Big)^2du+\frac{3}{(-K_X)^3}\int\limits_0^\tau\int\limits_0^{\infty}\mathrm{vol}\big(P(u)\big\vert_{S}-vC\big)dvdu,
$$
which we can rewrite as
$$
 S\big(W^S_{\bullet,\bullet};C\big)=\frac{3}{(-K_X)^3}\int\limits_0^{\tau}d(u)\Big(P(u)\big\vert_{S}\Big)^2du+\frac{3}{(-K_X)^3}\int\limits_0^\tau\int\limits_0^{t(u)}\big(P(u,v)\big)^2dvdu.
$$
If $Z$ is a~curve, $Z\subset S$ and $C=Z$, then it follows from \cite{AbbanZhuang,Book} that
\begin{equation}
\label{equation:Kento-curve}
\frac{A_X(\mathbf{F})}{S_X(\mathbf{F})}\geqslant\min\Bigg\{\frac{1}{S_X(S)},\frac{1}{S\big(W^S_{\bullet,\bullet};C\big)}\Bigg\}.
\end{equation}

Let $f\colon\widetilde{S}\to S$ be the~blow up of the~point $P$,
let $F$ be the~$f$-exceptional curve, let $\widetilde{N}^\prime(u)$~be~the~strict transform on $\widetilde{S}$ of the~$\mathbb{R}$-divisor $N(u)\vert_{S}$,
and let $\widetilde{d}(u)=\mathrm{mult}_P(N(u)\vert_S)$. Then
$$
f^*\big(N(u)\big\vert_S\big)=\widetilde{d}(u)F+\widetilde{N}^\prime(u).
$$
For every $u\in [0,\tau]$, set
$$
\widetilde{t}(u)=\sup\Big\{v\in \mathbb R_{\geqslant 0} \ \big|\ \text{the divisor $f^*\big(P(u)\big|_S\big)-vF$ is pseudo-effective}\Big\}.
$$
For $v\in [0, \widetilde{t}(u)]$, we let $\widetilde{P}(u,v)$ be the~positive part of the~Zariski decomposition of $f^*(P(u)|_S)-vF$,
and we let $\widetilde{N}(u,v)$ be its negative part. Let
$$
S\big(W^S_{\bullet,\bullet};F\big)=\frac{3}{(-K_X)^3}\int\limits_0^{\tau}\widetilde{d}(u)\Big(f^*\big(P(u)\big|_S\big)\Big)^2du+
\frac{3}{(-K_X)^3}\int\limits_0^\tau\int\limits_0^{\infty}\mathrm{vol}\big(f^*\big(P(u)\big|_S\big)-vF\big)dvdu.
$$
Then
$$
S\big(W^S_{\bullet,\bullet};F\big)=\frac{3}{(-K_X)^3}\int\limits_0^{\tau}\widetilde{d}(u)\big(\widetilde{P}(u,0)\big)^2du+
\frac{3}{20}\int\limits_0^\tau\int\limits_0^{\widetilde{t}(u)}\big(\widetilde{P}(u,v)\big)^2dvdu.
$$
For every point $O\in F$, we let
$$
S\big(W_{\bullet, \bullet,\bullet}^{\widetilde{S},F};O\big)=\frac{3}{(-K_X)^3}\int\limits_0^\tau\int\limits_0^{\widetilde{t}(u)}\big(\widetilde{P}(u,v)\cdot F\big)^2dvdu+F_O\big(W_{\bullet,\bullet,\bullet}^{\widetilde{S},F}\big)
$$
for
$$
F_O\big(W_{\bullet,\bullet,\bullet}^{\widetilde{S},F}\big)=\frac{6}{(-K_X)^3}
\int\limits_0^\tau\int\limits_0^{\widetilde{t}(u)}\big(\widetilde{P}(u,v)\cdot F\big)\cdot \mathrm{ord}_O\big(\widetilde{N}^\prime(u)\big|_F+\widetilde{N}(u,v)\big|_F\big)dvdu.
$$
Then it follows from \cite{AbbanZhuang,Book} that
\begin{equation}
\label{equation:Kento-point}
\frac{A_{X}(\mathbf{F})}{S_X(\mathbf{F})}\geqslant
\min\Bigg\{\frac{1}{S_X({S})},\frac{2}{S\big(W^{S}_{\bullet,\bullet};F\big)},\inf_{O\in F}\frac{1}{S\big(W_{\bullet, \bullet,\bullet}^{\widetilde{S},F};O\big)}\Bigg\}.
\end{equation}
Thus, if $S_X({S})<1$, $S(W^{S}_{\bullet,\bullet};F)<2$
and $S(W_{\bullet, \bullet,\bullet}^{\widetilde{S},F};O)<1$ for every point $O\in F$, then $\beta(\mathbf{F})>0$.

Now, we are ready to prove Theorem A.
We must show that $\beta(\mathbf{F})>0$ if $Z$ is not a~point in $E\cap E^\prime$.
If $Z$ is a~surface, it follows from \cite{Fujita2016} that $\beta(\mathbf{F})>0$.
Hence, we may assume that $Z$ is \textbf{not} a~surface.

\begin{lemma}[{cf. \cite{CheltsovPokora}}]
\label{lemma:E}
Suppose that $Z$ is a~curve, $Z\subset E$, and $\pi(Z)$ is not a~point. Then $\beta(\mathbf{F})>0$.
\end{lemma}

\begin{proof}
Let $e$ be the~invariant of the~ruled surface $E$ defined in Proposition 2.8 in \cite[Chapter V]{Hartshorne}.
Then $e\geqslant -3$ \cite{Nagata}. Moreover, there exists a~section $C_{0}$ of the~projection $E\to C_6$ such that $C_0^2=-e$.
Let $\ell$ a~fiber of this projection.
Then $H\vert_{E}\equiv 6\ell$ and $E\vert_{E}\equiv -C_{0} + \lambda \ell$ for some integer $\lambda$. Since
$$
-28=-c_{1}\big(N_{C_6/\mathbb{P}^{3}}\big)=E^{3}=(-C_{0}+\lambda \ell)^{2}=-e-2\lambda,
$$
we get $\lambda = \frac{28-e}{2}$, so $e$ is even and $e\geqslant -2$.
Since $H^\prime$ is nef and $H^\prime\vert_{E}\equiv C_{0}+(18-\lambda)\ell$, we get
$$
0\leqslant H^\prime\cdot C_0=\big(C_{0}+(18-\lambda)\ell\big)\cdot C_{0}=\frac{8-e}{2},
$$
which implies that $e\leqslant 8$. Thus, we see that $e\in\{-2,0,2,4,6,8\}$.

Set $S=E$ and $C=Z$.
Let us estimate $S(W^S_{\bullet,\bullet};C)$.
It follows from Example~\ref{example:AZ-surfaces-E} that $\tau=\frac{1}{2}$ and
$$
P(u)\big\vert_{S}\equiv\left\{\aligned
&(1+u)C_{0}+\frac{20+e+ue-28u}{2}\ell \ \text{ for } 0\leqslant u \leqslant \frac{1}{3}, \\
&(4-8u)C_{0}+2(1-2u)(8+e)\ell \ \text{ for } \frac{1}{3} \leqslant u\leqslant \frac{1}{2}.
\endaligned
\right.
$$
If $0\leqslant u \leqslant \frac{1}{2}$, then $N(u)=0$.
If $\frac{1}{2} \leqslant u\leqslant \frac{1}{3}$, then $N(u)\big\vert_{S}=(3u-1)E^\prime\vert_{S}$,
where $E^\prime\vert_{S}\equiv 3C_0+\frac{12+3e}{2}\ell$.
By Proposition 2.20 in \cite[Chapter~V]{Hartshorne}, we have $Z\equiv aC_{0}+b\ell$
for integers $a$ and $b$  such that $a\geqslant 0$ and $b\geqslant ae$.
Since $\pi(Z)$ is not a~point, we have $a\geqslant 1$.  Then $\mathrm{ord}_{C}(E^\prime\vert_{S})\leqslant 3$.
Hence, if $\frac{1}{3} \leqslant u\leqslant \frac{1}{2}$, then $d(u)\leqslant 3(3u-1)$.
This gives
\begin{multline*}
S(W_{\bullet,\bullet}^{{S}};{C})=
\frac{3}{20}\int\limits_{\frac{1}{3}}^{\frac{1}{2}}128(2u-1)^2d(u)du+
\frac{3}{20}\int\limits_0^{\frac{1}{2}}\int\limits_0^\infty \mathrm{vol}\big(P(u)\big\vert_{{S}}-v{C}\big)dvdu\leqslant \\
\leqslant\frac{3}{20}\int\limits_{\frac{1}{3}}^{\frac{1}{2}}384(3u-1)(2u-1)^2du+
\frac{3}{20}\int\limits_0^{\frac{1}{2}}\int\limits_0^\infty \mathrm{vol}\big(P(u)\big\vert_{{S}}-v{C}\big)dvdu = \\
=\frac{2}{45}+\frac{3}{20}\int\limits_0^{\frac{1}{2}}\int\limits_0^\infty \mathrm{vol}\big(P(u)\big\vert_{{S}}-v{C}\big)dvdu
=\frac{2}{45}+\frac{3}{20}\int\limits_0^{\frac{1}{2}}\int\limits_0^\infty \mathrm{vol}\big(P(u)\big\vert_{{S}}-v(aC_{0} + b\ell)\big)dvdu.
\end{multline*}
Thus, we conclude that
$S(W_{\bullet,\bullet}^{{S}};{C})\leqslant\frac{2}{45}+\frac{3}{20}\int\limits_0^{\frac{1}{2}}\int\limits_0^\infty \mathrm{vol}\big(P(u)\big\vert_{{S}}-v(aC_{0} + b\ell)\big)dvdu$.

Suppose that $b\geqslant 0$. Then
$$
\frac{3}{20}\int\limits_0^{\frac{1}{2}}\int\limits_0^\infty \mathrm{vol}\big(P(u)\big\vert_{{S}}-v(aC_{0} + b\ell)\big)dvdu\leqslant
\frac{3}{20}\int\limits_0^{\frac{1}{2}}\int\limits_{0}^{\infty}\mathrm{vol}\big(P(u)\vert_{S}-vC_0\big)dvdu.
$$
On the~other hand, we have
$$
P(u)\big\vert_{S}-vC_0\equiv\left\{\aligned
&(1+u-v)C_{0}+\frac{20+e+ue-28u}{2}\ell \ \text{if $0\leqslant u\leqslant \frac{1}{3}$}, \\
&(4-8u-v)C_{0}+2(1-2u)(8+e)\ell\ \text{if $\frac{1}{3}\leqslant u\leqslant \frac{1}{2}$}.
\endaligned
\right.
$$
Hence, if $0\leqslant u\leqslant \frac{1}{3}$, then the~divisor $P(u)\big\vert_{S}-vC_0$ is pseudoeffective $\iff$ it is nef $\iff$ $v\leqslant 1+u$.
Likewise, if $\frac{1}{3}\leqslant u\leqslant \frac{1}{2}$, then $P(u)\big\vert_{S}-vC_0$ is pseudoeffective $\iff$ it is nef $\iff$ $v\leqslant 4-8u$.
Then
\begin{multline*}
\quad\quad S(W_{\bullet,\bullet}^{{S}};{C})\leqslant\frac{2}{45}+\frac{3}{20}\int\limits_0^{\frac{1}{2}}\int\limits_{0}^{\infty}\mathrm{vol}\big(P(u)\vert_{S}-vC_0\big)dvdu=\\
=\frac{2}{45}+\frac{3}{20}\int\limits_{0}^{\frac{1}{3}}\int\limits_{0}^{1+u}\Bigg((1+u-v)C_0+\frac{20+e+ue-28u}{2}\ell\Bigg)^2dvdu+\\
+\frac{3}{20}\int\limits_{\frac{1}{3}}^{\frac{1}{2}}\int\limits_{0}^{4-8u}\big((4-8u-v)C_0+2(1-2u)(8+e)\ell\big)^2dvdu=\frac{23e}{1440}+\frac{221}{360}<1,\quad\quad\quad\quad
\end{multline*}
because $e\leqslant 8$. Then $\beta(\mathbf{F})>0$ by \eqref{equation:Kento-curve}, since we know from Example~\ref{example:AZ-surfaces-E} that $S_X(S)<1$.

Thus, to complete the~proof, we may assume that $b<0$. Then $e<0$, so that $e=-2$, since $b\geqslant ae$.
Hence, it follows from Proposition 2.21 in \cite[Chapter~V]{Hartshorne} that $a\geqslant 2$ and $b\geqslant -a$.
Then
$$
\frac{3}{20}\int\limits_0^{\frac{1}{2}}\int\limits_{0}^{\infty}\mathrm{vol}\big(P(u)\vert_{S}-vC\big)dvdu\leqslant
\frac{3}{20}\int\limits_0^{\frac{1}{2}}\int\limits_{0}^{\infty}\mathrm{vol}\big(P(u)\vert_{S}-v(2C_0-2\ell)\big)dvdu.
$$
Moreover, arguing as above, we compute
$$
\frac{3}{20}\int\limits_0^{\frac{1}{2}}\int\limits_{0}^{\infty}\mathrm{vol}\big(P(u)\vert_{S}-v(2C_0-2\ell)\big)dvdu=\frac{41}{144},
$$
which gives $S(W_{\bullet,\bullet}^{S};C)\leqslant\frac{2}{45}+\frac{41}{144}=\frac{79}{240}<1$,
so that $\beta(\mathbf{F})>0$ by \eqref{equation:Kento-curve}.
\end{proof}

Similarly, we prove that

\begin{lemma}
\label{lemma:E-prime}
Suppose that $Z$ is a~curve, $Z\subset E^\prime$, and $\pi^\prime(Z)$ is not a~point. Then $\beta(\mathbf{F})>0$.
\end{lemma}

Now, suppose that $Z$ is \textbf{not} a~point in $E\cap E^\prime$.
To prove Theorem A, we must show that $\beta(\mathbf{F})>0$.
Let $P$ be a~general point in $Z$. By Lemmas~\ref{lemma:E} and \ref{lemma:E-prime},
we may assume that either $P\not\in E$ or $P\not\in E^\prime$.
Hence, without loss of generality, we may assume that $P\not\in E$.
Let us show that $\beta(\mathbf{F})>0$.

Let $S$ be a~sufficiently general surface in $|H|$ that contains~$P$.
Then it follows from the~adjunction formula that $-K_S\sim H^\prime\vert_{S}$.
Set $\Pi=\pi(S)$. Then $\Pi$ is a~general plane in $\mathbb{P}^3$ that contains $\pi(P)$.
Write
$$
\Pi\cap C_6=\big\{P_1,P_2,P_3,P_4,P_5,P_6\big\},
$$
where $P_1$, $P_2$, $P_3$, $P_4$, $P_5$, $P_6$ are distinct points.
Then $\pi$ induces a~birational morphism~\mbox{$\varpi\colon S\to \Pi$}, which is a~blow up of the~intersection points $P_1$, $P_2$, $P_3$, $P_4$, $P_5$, $P_6$.

\begin{lemma}
\label{lemma:cubic-surface}
The divisor $-K_{S}$ is ample.
\end{lemma}

\begin{proof}
We must show that at most three points among $P_1$, $P_2$, $P_3$, $P_4$, $P_5$, $P_6$ are contained in a~line,
and not all of these six points are contained in an irreducible conic.

If there exists a~line $\ell\subset\Pi$ such that $\ell$ contains at least three points among $P_1$, $P_2$, $P_3$, $P_4$, $P_5$, $P_6$,
then $\ell$ is a~trisecant of the~curve~$C_6$, so that the~line $\ell$ is contained in $\pi(E^\prime)$,
and its strict transform on the~threefold $X$ is a~fiber of the~projection $E^\prime\to C_6^\prime$.
But the~planes in $\mathbb{P}^3$ containing $\pi(P)$ and a~trisecant of the~curve $C_6$ form a~one-dimensional family.
Hence, a~general plane in $\mathbb{P}^3$ that contains the~point $\pi(P)$ does not contain trisecants of the~curve $C_6$.
Therefore, we conclude that at most two points among $P_1$, $P_2$, $P_3$, $P_4$, $P_5$, $P_6$ are contained in a~line.

Similarly, if the~points $P_1$, $P_2$, $P_3$, $P_4$, $P_5$, $P_6$ are contained in an irreducible conic in $\Pi$,
then its strict transform on the~threefold $X$ has trivial intersection with $H^\prime\sim 3H-E$, which implies that this conic is the~image of a~fiber of the~projection $E^\prime\to C_6^\prime$,
which is impossible, since these fibers are mapped to lines in $\mathbb{P}^3$. Therefore, the~divisor $-K_S$ is ample.
\end{proof}

Thus, we can identify $S$ with a~smooth cubic surface in $\mathbb{P}^3$. Recall that $P\not\in E$.

\begin{lemma}
\label{lemma:cubic-surface-line}
Suppose that there exists a~line $\ell\subset S$ such that $P\in \ell$. Then $\pi(\ell)$ is a~conic.
\end{lemma}

\begin{proof}
If $\pi(\ell)$ is not a~conic, then $\pi(\ell)$ is a~secant of the~curve $C_6$ that contains $\pi(P)$.
Let us show that we can choose $\Pi$ such that it does not contain any secant of the~curve $C_6$.

Let $\phi\colon\mathbb{P}^3\dasharrow\mathbb{P}^2$ be the~linear projection from~$\pi(P)$.
Since $C_6$ is not hyperelliptic and $\pi(P)\not\in C_6$, one of the~following two possibilities holds:
\begin{enumerate}
\item $\phi(C_6)$ is a~singular curve of degree $6$, and $\phi$ induces a~birational morphism $C_6\to\phi(C_6)$,
\item $\phi(C_6)$ is a~smooth cubic, and $\phi$ induces a~double cover $C_6\to\phi(C_6)$.
\end{enumerate}
In the~second case, the~curve $C_6$ is contained in an irrational cubic cone in $\mathbb{P}^3$,
which is impossible, because the~composition $\pi^\prime\circ\pi^{-1}$ birationally maps every cubic surface containing $C_6$ to a~plane in~$\mathbb{P}^3$.
Thus, we see that $\phi(C_6)$ is a~singular irreducible curve of degree $6$.

All secants of the~curve $C_6$ containing $\pi(P)$ are mapped by $\phi$ to~singular points of the~curve $\phi(C_6)$.
Since this curve has finitely many singular points, there are finitely many secants of the~curve $C_6$ that pass through $\pi(P)$.
Hence, since $\Pi$ is a~general plane in $\mathbb{P}^3$ that contains~$\pi(P)$,
we may assume that it does not contain secants of the~curve $C_6$ containing $\pi(P)$, so  $\pi(\ell)$ is a~conic.
\end{proof}

Let $T$ be the~unique hyperplane section of the~surface $S\subset\mathbb{P}^3$ that is singular at $P$.
Then it follows from Lemma~\ref{lemma:cubic-surface-line} that either $P$ is not contained in any line in $S$,
and one of the~following cases holds:
\begin{enumerate}
\item[(a)] $T$ is an irreducible cubic curve that has a~node at $P$;

\item[(b)] $T$ is an irreducible cubic curve that has a~cusp at $P$;
\end{enumerate}
or $P$ is contained in a~unique line $\ell\subset S$, $\pi(\ell)$ is a~conic,
and one of the~following cases holds:

\begin{enumerate}
\item[(c)] $T=\ell+C_2$ for a~smooth conic $C_2$ that intersect $\ell$ transversally at $P$;

\item[(d)] $T=\ell+C_2$ for a~smooth conic $C_2$ that is tangent to $\ell$ at $P$.
\end{enumerate}

Let us construct another curve in $S$ that is also singular at $P$.
Namely, for each $i\in\{1,2,3,4,5,6\}$, let $\ell_i$ be the~proper transform on $S$ of the~unique line in $\Pi$ that passes through the~points $\pi(P)$ and~$P_i$.
Set $L=\ell_1+\ell_2+\ell_3+\ell_4+\ell_5+\ell_6$.
Then it follows from Example~\ref{example:AZ-surfaces-H} that
$$
P(u)\big\vert_{S}\sim_{\mathbb{R}}
\left\{\aligned
&\frac{2+u}{3}T+\frac{1-u}{3}L\ \text{if $0\leqslant u\leqslant 1$}, \\
&(4-3u)T\ \text{if $1\leqslant u\leqslant \frac{4}{3}$}.
\endaligned
\right.
$$
Recall from Example~\ref{example:AZ-surfaces-H} that $\tau=\frac{4}{3}$ and $S_X(S)=\frac{53}{120}$.

Let $\widetilde{T}$ and $\widetilde{L}$ be the~proper transforms on $\widetilde{S}$ of the~curves $T$ and $L$, respectively.
If $0\leqslant u\leqslant 1$, then
$$
f^*\big(P(u)\vert_{S}\big)-vF\sim_{\mathbb{R}}\frac{2+u}{3}\widetilde{T}+\frac{1-u}{3}\widetilde{L}+\frac{10-4u-3v}{3}F,
$$
which implies that $\widetilde{t}(u)=\frac{10-4u}{3}$.
Similarly, if $1\leqslant u\leqslant\frac{4}{3}$, then
$$
f^*\big(P(u)\vert_{S}\big)-vF\sim_{\mathbb{R}}(4-3u)\widetilde{T}+(8-6u-v)F,
$$
which implies that  $\widetilde{t}(u)=8-6u$.

Finally, set $R=E^\prime\vert_{S}$. Then $R$ is a~smooth curve.
Let $\widetilde{R}$ be its strict transform on $\widetilde{S}$.
Then
$$
\widetilde{N}^\prime(u)=
\left\{\aligned
&0\ \text{if $0\leqslant u\leqslant 1$}, \\
&(u-1)\widetilde{R} \ \text{if $1\leqslant u\leqslant \frac{4}{3}$}.
\endaligned
\right.
$$
So, if $0\leqslant u\leqslant 1$ or $P\not\in E^\prime$, then $\widetilde{d}(u)=0$.
Similarly, if $1\leqslant u\leqslant \frac{4}{3}$ and $P\in E^\prime$, then $\widetilde{d}(u)=u-1$.

\begin{lemma}
\label{lemma:dP3-blow-up}
Suppose that $P$ is not contained in any line in $S$.
Then $\beta(\mathbf{F})>0$.
\end{lemma}

\begin{proof}
The curve $T$ is irreducible. If $0\leqslant u\leqslant 1$, then
$$
P(u,v)\sim_{\mathbb{R}}
\left\{\aligned
&\frac{2+u}{3}\widetilde{T}+\frac{1-u}{3}\widetilde{L}+\frac{10-4u-3v}{3}F\ \text{if $0\leqslant v\leqslant \frac{6-3u}{2}$}, \\
&\frac{20-8u-6v}{3}\widetilde{T}+\frac{1-u}{3}\widetilde{L}+\frac{10-4u-3v}{3}F\ \text{if $\frac{6-3u}{2}\leqslant v\leqslant 3-u$}, \\
&\frac{10-4u-3v}{3}\big(2\widetilde{T}+\widetilde{L}+F\big)\ \text{if $3-u\leqslant v\leqslant \frac{10-4u}{3}$},
\endaligned
\right.
$$
and
$$
N(u,v)=\left\{\aligned
&0\ \text{if $0\leqslant v\leqslant \frac{6-3u}{2}$}, \\
&(2v-6+3u)\widetilde{T}\ \text{if $\frac{6-3u}{2}\leqslant v\leqslant 3-u$}, \\
&(2v-6+3u)\widetilde{T}+(v+u-3)\widetilde{L}\ \text{if $3-u\leqslant v\leqslant \frac{10-4u}{3}$}.
\endaligned
\right.
$$
This gives
$$
\big(P(u,v)\big)^2=
\left\{\aligned
&u^2-v^2-8u+10\ \text{if $0\leqslant v\leqslant \frac{6-3u}{2}$}, \\
&10u^2+12uv+3v^2-44u-24v+46\ \text{if $\frac{6-3u}{2}\leqslant v\leqslant 3-u$}, \\
&(10-4u-3v)^2\ \text{if $3-u\leqslant v\leqslant \frac{10-4u}{3}$}
\endaligned
\right.
$$
and
$$
P(u,v)\cdot F=
\left\{\aligned
&v\ \text{if $0\leqslant v\leqslant \frac{6-3u}{2}$}, \\
&12 - 6u - 3v\ \text{if $\frac{6-3u}{2}\leqslant v\leqslant 3-u$}, \\
&30 - 12u - 9v\ \text{if $3-u\leqslant v\leqslant \frac{10-4u}{3}$.}
\endaligned
\right.
$$
Similarly, if $1\leqslant u\leqslant \frac{4}{3}$, then
$$
P(u,v)\sim_{\mathbb{R}}\left\{\aligned
&(4-3u)\widetilde{T}+(8-6u-v)F\ \text{if $0\leqslant v\leqslant\frac{12-9u}{2}$}, \\
&(8-6u-v)\big(2\widetilde{T}+F\big)\ \text{if $\frac{12-9u}{2}\leqslant v\leqslant 8-6u$},
\endaligned
\right.
$$
and
$$
N(u,v)=\left\{\aligned
&0\ \text{if $0\leqslant v\leqslant\frac{12-9u}{2}$}, \\
&(2v+9u-12)\widetilde{T}\ \text{if $\frac{12-9u}{2}\leqslant v\leqslant 8-6u$}.
\endaligned
\right.
$$
This gives
$$
\big(P(u,v)\big)^2=
\left\{\aligned
&27u^2-v^2-72u+48\ \text{if $0\leqslant v\leqslant\frac{12-9u}{2}$}, \\
&3(8-6u-v)^2\ \text{if $\frac{12-9u}{2}\leqslant v\leqslant 8-6u$},
\endaligned
\right.
$$
and
$$
P(u,v)\cdot F=
\left\{\aligned
&v\ \text{if $0\leqslant v\leqslant\frac{12-9u}{2}$}, \\
&24-18u-3v\ \text{if $\frac{12-9u}{2}\leqslant v\leqslant 8-6u$}.
\endaligned
\right.
$$

Thus, if $P\in E^\prime$, then
\begin{multline*}
 S\big(W_{\bullet,\bullet}^S;F\big)=\frac{3}{20}\int\limits_1^{\frac{4}{3}}(27u^2-72u+48)(u-1)du+\frac{3}{20}\int\limits_{0}^{1}\int\limits_{0}^{\frac{6-3u}{2}}u^2-v^2-8u+10dvdu+\\
+\frac{3}{20}\int\limits_{0}^{1}\int\limits_{\frac{6-3u}{2}}^{3-u}10u^2 + 12uv + 3v^2 - 44u - 24v + 46dvdu+\frac{3}{20}\int\limits_{0}^{1}\int\limits_{3-u}^{\frac{10-4u}{3}}(4u+3v-10)^2dvdu+\\
 +\frac{3}{20}\int\limits_{1}^{\frac{4}{3}}\int\limits_{0}^{\frac{12-9u}{2}}27u^2-v^2-72u+48dvdu+\frac{3}{20}\int\limits_{1}^{\frac{4}{3}}\int\limits_{\frac{12-9u}{2}}^{8-6u}3(6u+v-8)^2dvdu=\frac{41}{24}<2.
\end{multline*}
Similarly, if $P\not\in E^\prime$, then $S(W_{\bullet,\bullet}^S;F)=\frac{409}{240}<\frac{41}{24}<2$.

Now, let $O$ be a~point in $F$. Let us compute $S(W_{\bullet,\bullet,\bullet}^{\widetilde{S},F};O)$.
We have
\begin{multline*}
 S(W_{\bullet,\bullet,\bullet}^{\widetilde{S},F};O)=\frac{3}{20}\int\limits_{0}^{1}\int\limits_{0}^{\frac{6-3u}{2}}v^2dvdu+\frac{3}{20}\int\limits_{0}^{1}\int\limits_{\frac{6-3u}{2}}^{3-u}(12-6u-3v)^2dvdu+\\
 +\frac{3}{20}\int\limits_{0}^{1}\int\limits_{3-u}^{\frac{10-4u}{3}}(30-12u-9v)^2dvdu+\frac{3}{20}\int\limits_{1}^{\frac{4}{3}}\int\limits_{0}^{\frac{12-9u}{2}}8v^2dvdu+\frac{3}{20}\int\limits_{1}^{\frac{4}{3}}\int\limits_{\frac{12-9u}{2}}^{8-6u}(24-18u-3v)^2dvdu+F_O\big(W_{\bullet,\bullet,\bullet}^{\widetilde{S},F}\big),
\end{multline*}
so that $S(W_{\bullet,\bullet,\bullet}^{\widetilde{S},F};O)=\frac{63}{80}+F_O(W_{\bullet,\bullet,\bullet}^{\widetilde{S},F})$.
In particular, if $P\not\in E^\prime$ and $O\not\in\widetilde{T}\cup\widetilde{C}$,
then $F_O(W_{\bullet,\bullet,\bullet}^{\widetilde{S},F})=0$, which implies that $S(W_{\bullet,\bullet,\bullet}^{\widetilde{S},F};O)=\frac{63}{80}$.
Let us compute $F_O(W_{\bullet,\bullet,\bullet}^{\widetilde{S},F})$ in the~remaining cases.

First, we deal with the~case $P\not\in E^\prime$. If $P\not\in E^\prime$, then we have $O\not\in\mathrm{Supp}(\widetilde{N}^\prime(u))$ for every $u\in[0,\frac{4}{3}]$.
Moreover, if $P\not\in E^\prime$ and $O\in\widetilde{L}$, then $O\not\in\widetilde{T}$,
and $\widetilde{L}$ intersects $F$ transversally at $O$, which gives
$$
 S\big(W_{\bullet,\bullet,\bullet}^{\widetilde{S},F};O\big)=\frac{63}{80}+\frac{6}{20}\int\limits_0^{1}\int\limits_{3-u}^{\frac{10-4u}{3}}\big(P(u,v)\cdot F\big)(v+u-3)\big(\widetilde{L}\cdot F\big)_Odvdu=\frac{19}{24}.
$$
Similarly, if $P\not\in E^\prime$ and $O\in\widetilde{T}$, then  $O\not\in\widetilde{L}$ and
\begin{multline*}
 S\big(W_{\bullet,\bullet,\bullet}^{\widetilde{S},F};O\big)=\frac{63}{80}+\frac{6}{20}\int\limits_0^{1}\int\limits_{\frac{6-3u}{2}}^{\frac{10-4u}{3}}\big(P(u,v)\cdot F\big)(2v-6+3u)\big(\widetilde{T}\cdot F\big)_Odvdu+\\
+\frac{6}{20}\int\limits_1^{\frac{4}{3}}\int\limits_{\frac{12-9u}{2}}^{8-6u}\big(P(u,v)\cdot F\big)(2v+9u-12)\big(\widetilde{T}\cdot F\big)_O=\frac{63}{80}+\frac{6}{20}\int\limits_0^{1}\int\limits_{\frac{6-3u}{2}}^{3-u}(12-6u-3v)(2v-6+3u)\big(\widetilde{T}\cdot F\big)_Odvdu+\\
+\frac{6}{20}\int\limits_0^{1}\int\limits_{3-u}^{\frac{10-4u}{3}}(30-12u-9v)(2v-6+3u)\big(\widetilde{T}\cdot F\big)_Odvdu+\frac{6}{20}\int\limits_1^{\frac{4}{3}}\int\limits_{\frac{12-9u}{2}}^{8-6u}(24-18u-3)(2v+9u-12)\big(\widetilde{T}\cdot F\big)_O,
\end{multline*}
so $S(W_{\bullet,\bullet,\bullet}^{\widetilde{S},F};O)=\frac{63}{80}+\frac{5}{96}\big(\widetilde{T}\cdot F\big)_O\leqslant\frac{63}{80}+\frac{5}{96}\widetilde{T}\cdot F=\frac{107}{120}$.
Hence, if $P\not\in E^\prime$, then $\beta(\mathbf{F})>0$ by \eqref{equation:Kento-point}.

Therefore, to complete the~proof of the~lemma, we may assume that $P\in E^\prime$.
Since  $R$ is smooth, the~curve $\widetilde{R}$ intersects $F$ transversally at one point, so that
$$
\mathrm{ord}_O\big(\widetilde{N}^\prime(u)\big\vert_{F}\big)=
\left\{\aligned
&0\ \text{if $0\leqslant u\leqslant 1$}, \\
&0\ \text{if $1\leqslant u\leqslant \frac{4}{3}$ and $O\ne \widetilde{R}\cap F$}, \\
&u-1\ \text{if $1\leqslant u\leqslant \frac{4}{3}$ and $O=\widetilde{R}\cap F$}.
\endaligned
\right.
$$
Hence, if $O\ne \widetilde{R}\cap F$, then $S(W_{\bullet,\bullet,\bullet}^{\widetilde{S},F};O)$ can be computed as in the~case $P\not\in E^\prime$.
Thus, we may also assume that $O=\widetilde{R}\cap F$.
Moreover, if $O\in\widetilde{L}$, then our previous calculations give
\begin{multline*}
\quad\quad\quad \quad S\big(W_{\bullet,\bullet,\bullet}^{\widetilde{S},F};O\big)=
\frac{6}{20}\int\limits_1^{\frac{4}{3}}\int\limits_0^{\widetilde{t}(u)}\big(\widetilde{P}(u,v)\cdot F\big)(u-1)dvdu+\frac{19}{24}=\\
=\frac{6}{20}\int\limits_1^{\frac{4}{3}}\int\limits_0^{\frac{12-9u}{2}}v(u-1)dvdu+
\frac{6}{20}\int\limits_1^{\frac{4}{3}}\int\limits_{\frac{12-9u}{2}}^{8-6u}(24-18u-3v)(u-1)dvdu+\frac{19}{24}=\frac{191}{240}.
\end{multline*}
Similarly, if $O\in\widetilde{T}$, then, using our previous computations, we get
$$
S\big(W_{\bullet,\bullet,\bullet}^{\widetilde{S},F};O\big)=\frac{1}{241}+\frac{63}{80}+\frac{5}{96}\big(\widetilde{T}\cdot F\big)_O
\leqslant\frac{1}{241}+\frac{63}{80}+\frac{5}{96}\widetilde{T}\cdot F=\frac{43}{48}.
$$
Thus, we see that $S(W_{\bullet,\bullet,\bullet}^{\widetilde{S},F};O)<1$ for every point $O\in F$,
so that $\beta(\mathbf{F})>0$ by \eqref{equation:Kento-point}.
\end{proof}

To complete the~proof of Theorem A, we may assume that $T=\ell+C_2$ and $P\in\ell\cap C_2$,
where $\ell$ is a~line such that $\pi(\ell)$ is a~conic in $\mathbb{P}^2$,
and $C_2$ is a~smooth conic such that $\pi(C_2)$ is a~line.
Then $C_2$ is one of the~curves $\ell_1$, $\ell_2$, $\ell_3$, $\ell_4$, $\ell_5$, $\ell_6$,
so we may assume that $C_2=\ell_6$.
Set $L^\prime=\ell_1+\ell_2+\ell_3+\ell_4+\ell_5$.
Let us denote by $\widetilde{\ell}$, $\widetilde{C}_2$, $\widetilde{L}^\prime$ the~strict transforms on the~surface $\widetilde{S}$ of the~curves $\ell$, $C_2$, $L^\prime$, respectively.
Then $\widetilde{\ell}\cap\widetilde{L}^\prime=\varnothing$ and $\widetilde{C}_2\cap\widetilde{L}^\prime=\varnothing$.
Moreover, if $0\leqslant u\leqslant 1$, then
$$
P(u,v)\sim_{\mathbb{R}}
\left\{\aligned
&\frac{2+u}{3}\widetilde{\ell}+\widetilde{C}_2+\frac{1-u}{3}\widetilde{L}^\prime+\frac{10-4u-3v}{3}F\ \text{if $0\leqslant v\leqslant 3-2u$}, \\
&\frac{13-4u-3v}{6}\widetilde{\ell}+\widetilde{C}_2+\frac{1-u}{3}\widetilde{L}^\prime+\frac{10-4u-3v}{3}F\ \text{if $3-2u\leqslant v\leqslant \frac{9-4u}{3}$}, \\
&\frac{10-4u-3v}{3}\big(2\widetilde{\ell}+3\widetilde{C}_2+F\big)+\frac{1-u}{3}\widetilde{L}^\prime\ \text{if $\frac{9-4u}{3}\leqslant v\leqslant 3-u$}, \\
&\frac{10-4u-3v}{3}\big(2\widetilde{\ell}+\widetilde{L}^\prime+3\widetilde{C}_2+F\big)\ \text{if $3-u\leqslant v\leqslant \frac{10-4u}{3}$},
\endaligned
\right.
$$
and
$$
N(u,v)=\left\{\aligned
&0\ \text{if $0\leqslant v\leqslant 3-2u$}, \\
&\frac{v+2u-3}{2}\widetilde{\ell}\ \text{if $3-2u\leqslant v\leqslant \frac{9-4u}{3}$}, \\
&(2v+3u-6)\widetilde{\ell}+(3v+4u-9)\widetilde{C}_2\ \text{if $\frac{9-4u}{3}\leqslant v\leqslant 3-u$}, \\
&(2v+3u-6)\widetilde{\ell}+(3v+4u-9)\widetilde{C}_2+(v+u-3)\widetilde{L}^\prime\ \text{if $3-u\leqslant v\leqslant \frac{10-4u}{3}$}.
\endaligned
\right.
$$
This gives
$$
\big(P(u,v)\big)^2=
\left\{\aligned
&u^2-v^2-8u+10\ \text{if $0\leqslant v\leqslant 3-2u$}, \\
&\frac{29}{2}-14u-3v+3u^2-\frac{v^2}{2}+2vu\ \text{if $3-2u\leqslant v\leqslant \frac{9-4u}{3}$}, \\
&11u^2+14uv+4v^2-50u-30v+55\ \text{if $\frac{9-4u}{3}\leqslant v\leqslant 3-u$}, \\
&(10-4u-3v)^2\ \text{if $3-u\leqslant v\leqslant \frac{10-4u}{3}$},
\endaligned
\right.
$$
and
$$
P(u,v)\cdot F=
\left\{\aligned
&v\ \text{if $0\leqslant v\leqslant 3-2u$}, \\
&\frac{3}{2}-u+\frac{v}{2}\ \text{if $3-2u\leqslant v\leqslant \frac{9-4u}{3}$}, \\
&15-7u-4v\ \text{if $\frac{9-4u}{3}\leqslant v\leqslant 3-u$}, \\
&30-12u-9v\ \text{if $3-u\leqslant v\leqslant \frac{10-4u}{3}$}.
\endaligned
\right.
$$
Furthermore, if $1\leqslant u\leqslant \frac{4}{3}$, then
$$
P(u,v)\sim_{\mathbb{R}}\left\{\aligned
&(4-3u)(\widetilde{\ell} + \widetilde{C}_2)+(8-6u-v)F\ \text{if $0\leqslant v\leqslant 4-3u$}, \\
&\frac{12-9u-v}{2}\widetilde{\ell}+(4-3u)\widetilde{C}_2+(8-6u-v)F\ \text{if $4-3u\leqslant v\leqslant \frac{20-15u}{3}$}, \\
&(8-6u-v)\big(2\widetilde{\ell}+3\widetilde{C}_2+F\big)\ \text{if $\frac{20-15u}{3}\leqslant v\leqslant 8-6u$},
\endaligned
\right.
$$
and
$$
N(u,v)=\left\{\aligned
&0\ \text{if $0\leqslant v\leqslant 4-3u$}, \\
&\frac{v+3u-4}{2}\widetilde{\ell}\ \text{if $4-3u\leqslant v\leqslant \frac{20-15u}{3}$}, \\
&(9u+2v-12)\widetilde{\ell}+(15u+3v-20)\widetilde{C}_2\ \text{if $\frac{20-15u}{3}\leqslant v\leqslant 8-6u$}.
\endaligned
\right.
$$
This gives
$$
\big(P(u,v)\big)^2=
\left\{\aligned
&27u^2-v^2-72u+48\ \text{if $0\leqslant v\leqslant 4-3u$}, \\
&56-84u-4v+\frac{63}{2}u^2-\frac{v^2}{2}+3vu\ \text{if $4-3u\leqslant v\leqslant \frac{20-15u}{3}$}, \\
&4(8-6u-v)^2\ \text{if $\frac{20-15u}{3}\leqslant v\leqslant 8-6u$},
\endaligned
\right.
$$
and
$$
P(u,v)\cdot F=
\left\{\aligned
&v\ \text{if $0\leqslant v\leqslant 4-3u$}, \\
&2-\frac{3u}{2}+\frac{v}{2}\ \text{if $4-3u\leqslant v\leqslant \frac{20-15u}{3}$}, \\
&32-24u-4v\ \text{if $\frac{20-15u}{3}\leqslant v\leqslant 8-6u$}.
\endaligned
\right.
$$
Now, as in the~proof of Lemma~\ref{lemma:dP3-blow-up}, we compute
$$
S\big(W_{\bullet,\bullet}^S;F\big)=\left\{\aligned
&\frac{77}{45}\ \text{if $P\in E^\prime$}, \\
&\frac{1229}{720}\ \text{if $P\not\in E^\prime$}.
\endaligned
\right.
$$
Similarly, if $O$ is a~point in $F$, we can compute $S(W_{\bullet,\bullet,\bullet}^{\widetilde{S},F};O)$ as we did this in
the proof of Lemma~\ref{lemma:dP3-blow-up}. The results of these computations are presented in the~following two tables:

\begin{center}
\renewcommand\arraystretch{1.5}
\begin{tabular}{|c|c|c|c|c|c|}
\hline
condition                                                  & $O\in\widetilde{\ell}\cap\widetilde{C}_2\cap\widetilde{R}$ & $\widetilde{\ell}\cap\widetilde{C}_2\ni O\not\in\widetilde{R}$ & $\widetilde{\ell}\cap\widetilde{R}\ni O\not\in\widetilde{C}_2$ & $\widetilde{\ell}\ni O\not\in\widetilde{R}\cup\widetilde{C}_2$ & $\widetilde{C}_2\cap\widetilde{R}\ni O\not\in\widetilde{\ell}$ \\
\hline
\hline
$S\big(W_{\bullet,\bullet,\bullet}^{\widetilde{S},F};O\big)$ & $\frac{163}{180}$ & $\frac{649}{720}$ & $\frac{1859}{2160}$ & $\frac{185}{216}$ & $\frac{1801}{2160}$ \\
\hline
\end{tabular}
\end{center}

\begin{center}
\renewcommand\arraystretch{1.5}
\begin{tabular}{|c|c|c|c|c|c|}
\hline
condition & $\widetilde{C}_2\ni O\not\in\widetilde{R}\cup\widetilde{\ell}$ & $O\in\widetilde{L}^\prime\cap\widetilde{R}$ & $\widetilde{L}^\prime\ni O\not\in\widetilde{R}$ & $\widetilde{R}\ni O\not\in\widetilde{\ell}\cup\widetilde{C}^\prime\cap\widetilde{L}^\prime$ & $O\not\in\widetilde{\ell}\cup\widetilde{C}^\prime\cap\widetilde{L}^\prime\cup\widetilde{R}$ \\
\hline
\hline
$S\big(W_{\bullet,\bullet,\bullet}^{\widetilde{S},F};O\big)$ & $\frac{112}{135}$ & $\frac{571}{720}$ & $\frac{71}{90}$ & $\frac{71}{90}$ & $\frac{113}{144}$ \\
\hline
\end{tabular}
\end{center}

\medskip

Thus, we proved that $S(W_{\bullet,\bullet}^S;F)<2$, and we proved that $S(W_{\bullet,\bullet,\bullet}^{\widetilde{S},F};O)<1$ for every point $O\in F$.
Therefore, using \eqref{equation:Kento-point}, we get $\beta(\mathbf{F})>0$. This completes the~proof of Theorem A.

\section{The proof of Theorem B}
\label{section:Aut}

Let us use all assumptions and notations introduced in Section~\ref{section:intro}.
Recall that
$$
\mathrm{Aut}\big(\mathbb{P}^3,C_6\big)\simeq\mathrm{Aut}\big(C,[D]\big)\subset\mathrm{Aut}(C),
$$
and all possibilities for the~group $\mathrm{Aut}(C)$ are listed in \cite{Bars,DolgachevBook},
where the~two lists disagree a~little~bit.
Moreover, since $\pi$ is $\mathrm{Aut}(\mathbb{P}^3,C_6)$-equivariant,
we can identify $\mathrm{Aut}(\mathbb{P}^3,C_6)$ with a~subgroup in $\mathrm{Aut}(X)$.
Then the action of the group $\mathrm{Aut}(X)$ on the set $\{E,E^\prime\}$ gives a~monomorphism
$$
\mathrm{Aut}(X)/\mathrm{Aut}\big(\mathbb{P}^3,C_6\big)\hookrightarrow\mumu_2,
$$
which is surjective if and only if $\mathrm{Aut}(X)$ has an element that swaps the~surfaces $E$ and $E^\prime$.

\begin{remark}[{\cite[Example~7.2.6]{DolgachevBook}}]
\label{remark:symmetric}
We can choose $M_1$, $M_2$, $M_3$ in \eqref{equation:P3-P3} to be symmetric $\iff$ $2D\sim K_C$.
Moreover, if  $M_1$, $M_2$, $M_3$ are symmetric, then $X$ admits the~involution
$$
\big([x_0:x_1:x_2:x_3],[y_0:y_1:y_2:y_3]\big)\mapsto\big([y_0:y_1:y_2:y_3],[x_0:x_1:x_2:x_3]\big).
$$
In this case, we have $\mathrm{Aut}(X)\simeq\mathrm{Aut}(\mathbb{P}^3,C_6)\times\mumu_2$.
For more details, see \cite{Ottaviani}.
\end{remark}

\begin{remark}[Kuznetsov]
\label{remark:Kuznetsov}
Set $V=H^0(\mathcal{O}_C(K_C+D))$, $W=H^0(\mathcal{O}_C(2K_C-D))$ and $G=\mathrm{Aut}(C,[D])$.
Let $\widehat{G}$ be a~central extension of the~group $G$ such that $D$ (considered as a line bundle) is $\widehat{G}$-linearizable.
Then the~sheaf $\mathcal{O}_C(D)$ admits a~$\widehat{G}$-equivariant resolvent
$$
0\rightarrow W^*\otimes \mathcal{O}_{\mathbb{P}^2}(-2)\rightarrow V\otimes \mathcal{O}_{\mathbb{P}^2}(-1)\rightarrow \mathcal{O}_C(D)\rightarrow 0,
$$
which is known as the~Beilinson resolvent.
Since $W^*\otimes\mathcal{O}_{\mathbb{P}^2}(-2) \to  V\otimes  \mathcal{O}_{\mathbb{P}^2}(-1)$ is $\widehat{G}$-equivariant,
the~corresponding map $\rho:V^*\otimes W^* \to H^0\big(\mathcal{O}_{\mathbb{P}^2}(1)\big)$ is equivariant,
where $H^0\big(\mathcal{O}_{\mathbb{P}^2}(1)\big)\simeq H^0(\mathcal{O}_{C}(K_C)\big)$ as $\widehat{G}$-representations.
On the~other hand, the~embedding $X\hookrightarrow\mathbb{P}^3\times\mathbb{P}^3$ given by \eqref{equation:P3-P3} can be realized~as
$$
X=\big(\mathbb{P}(V^*)\times\mathbb{P}(W^*)\big)\cap\mathbb{P}\big(\mathrm{ker}(\rho)\big),
$$
and the~$\widehat{G}$-action on $X$ factors through $G$, which is the~natural $G$-action.
\end{remark}

This remark gives

\begin{lemma}
\label{lemma:Aut-C2}
There exists a group homomorphism $\eta\colon\mathrm{Aut}(X)\to\mathrm{Aut}(C)$ such that
its restriction to the~subgroup $\mathrm{Aut}(\mathbb{P}^3,C_6)\simeq\mathrm{Aut}(C,[D])$ gives a natural embedding $\mathrm{Aut}(\mathbb{P}^3,C_6)\hookrightarrow\mathrm{Aut}(C)$.
\end{lemma}

\begin{proof}
Let $\mathcal{M}$ be the~two-dimensional linear system of divisors of degree $(1,1)$ on $\mathbb{P}^3\times\mathbb{P}^3$ that contains the~threefold $X$.
Then $\mathcal{M}$ can be identified with the~projectivization of the~three-dimensional vector space spanned by the~matrices $M_1$, $M_2$, $M_3$, which we will identify with $\mathbb{P}^2_{x,y,z}$.
Then $\mathrm{Aut}(X)$ naturally acts on this $\mathbb{P}^2_{x,y,z}$, because the~action of the~group $\mathrm{Aut}(X)$ on $X$ lifts to its action on $\mathbb{P}^3\times\mathbb{P}^3$.

Moreover, the~$\mathrm{Aut}(X)$-action on $\mathbb{P}^2_{x,y,z}$ preserves the~quartic curve in $\mathbb{P}^2_{x,y,z}$ given by
$$
\mathrm{det}\big(xM_1+yM_2+zM_2\big)=0,
$$
which parametrizes singular divisors in $\mathcal{M}$.
This curve is isomorphic to the~curve $C$, which gives us the~required homomorphism of groups $\eta\colon\mathrm{Aut}(X)\to\mathrm{Aut}(C)$.
It follows from Remark~\ref{remark:Kuznetsov} that this group homomorphism is functorial,
so it gives a~natural embedding $\mathrm{Aut}(\mathbb{P}^3,C_6)\hookrightarrow\mathrm{Aut}(C)$.
\end{proof}

\begin{corollary}
\label{corollary:Aut-1}
Either $\mathrm{Aut}(X)\simeq\mathrm{Aut}(\mathbb{P}^3,C_6)\times\mumu_2$ or $\mathrm{Aut}(X)$ is isomorphic to a~subgroup $\mathrm{Aut}(C)$.
\end{corollary}

Now, we are ready to state a~criterion when $\mathrm{Aut}(X)\ne\mathrm{Aut}(\mathbb{P}^3,C_6)$.

\begin{lemma}
\label{lemma:Aut}
$\mathrm{Aut}(X)\ne\mathrm{Aut}(\mathbb{P}^3,C_6)$ $\iff$ there is $g\in\mathrm{Aut}(C)$ such that $g^*(D)\sim K_C-D$.
\end{lemma}

\begin{proof}
By Remark~\ref{remark:Kuznetsov}, the~left copy of $\mathbb{P}^3$ in~\eqref{equation:diagram} can be be identified with $\mathbb{P}(H^0(\mathcal{O}_{C}(K_C+D))^\vee)$,
while the~right copy of $\mathbb{P}^3$ can be be identified with $\mathbb{P}(H^0(\mathcal{O}_{C}(2K_C-D))^\vee)$.
Thus, if $\mathrm{Aut}(C)$ contains an automorphism $g$ such that $g^*(D)\sim K_C-D$,
we can use it to identify both copies of $\mathbb{P}^3$ in~\eqref{equation:diagram},
which will give us an~automorphism of $X$ that swaps exceptional surfaces of the~blow ups $\pi$ and $\pi^\prime$.

Vice versa, if the~group $\mathrm{Aut}(X)$ is larger than $\mathrm{Aut}(\mathbb{P}^3,C_6)$,
it follows from the~proof of Lemma~\ref{lemma:Aut-C2} that there exists $g\in\mathrm{Aut}(C)$ such that $g^*(D)\sim K_C-D$.
\end{proof}

Recall that $\mathrm{Aut}(\mathbb{P}^3,C_6)\simeq\mathrm{Aut}(C,[D])$, where $D$ is a~divisor on $C$ of degree $2$ that satisfies \eqref{equation:very-ample}.
Using Remark~\ref{remark:symmetric}, Lemma~\ref{lemma:Aut} and its proof, we obtain

\begin{corollary}
\label{corollary:Aut}
One of the~following three cases holds:
\begin{itemize}
\item $2D\sim K_C$ and $\mathrm{Aut}(X)\simeq\mathrm{Aut}(C,[D])\times\mumu_2$,
\item $2D\not\sim K_C$, there is $g\in\mathrm{Aut}(C)$ such that $g^*(D)\sim K_C-D$, and
$$
\mathrm{Aut}(X)\simeq\big\langle\mathrm{Aut}(C,[D]),g\big\rangle.
$$
\item $\mathrm{Aut}(X)\simeq\mathrm{Aut}(C,[D])$, and $g^*(D)\not\sim K_C-D$ for every $g\in\mathrm{Aut}(C)$.
\end{itemize}
\end{corollary}

\begin{corollary}
\label{corollary:Aut-2}
If $\mathrm{Aut}(X)$ is not isomorphic to any subgroup of $\mathrm{Aut}(C)$, then $2D\sim K_C$.
\end{corollary}

Using Corollary~\ref{corollary:Aut}, we can find all possibilities for $\mathrm{Aut}(X)$,
but this requires a lot of work, because we have to analyze $\mathrm{Pic}^G(C)$ for every subgroup $G\subset\mathrm{Aut}(C)$.
This can be done using

\begin{proposition}[{\cite{Dolgachev1999}}]
\label{proposition:Dolgachev}
Let $G$ be a subgroup in $\mathrm{Aut}(C)$.
Then there exists exact sequence
$$
1\rightarrow \mathrm{Hom}\big(G,\mathbb{C}^*\big)\rightarrow\mathrm{Pic}\big(G,C\big)\rightarrow\mathrm{Pic}^G\big(C\big)\rightarrow H^2\big(G,\mathbb{C}^*\big)\rightarrow 1,
$$
where $\mathrm{Pic}(G,C)$ is the~group of $G$-linearized line bundles on $C$ modulo $G$-equivariant isomorphisms.
\end{proposition}

\noindent and

\begin{remark}
\label{remark:Dolgachev}
Let $G$ be a subgroup in $\mathrm{Aut}(C)$, let $\Sigma_1,\ldots,\Sigma_n$ be all $G$-orbits in $C$ of length less that~$|G|$.
We may assume that $|\Sigma_i|\geqslant |\Sigma_j|$ for $i\geqslant j$.
For every $i\in\{1,\ldots,n\}$, set
$$
e_i=\frac{|G|}{|\Sigma_i|}=\text{the order of the~stabilizer in $G$ of a point in $\Sigma_i$}.
$$
The~\emph{signature} of the~$G$-action on $C$ is the~tuple $\big[g;e_1,\ldots,e_n\big]$,
where $g$ is the~genus of the~curve~$C/G$. If $C/G\simeq\mathbb{P}^1$, then it follows from \cite{Dolgachev1999} that
$$
\mathrm{Pic}\big(G,C\big)\simeq\mathbb{Z}\oplus\mumu_{a_1}\oplus\mumu_{a_2}\oplus\cdots\oplus\mumu_{a_{n-1}}
$$
for $a_1=d_1, a_2=\frac{d_2}{d_1}, \ldots, a_{n-1}=\frac{d_{n-1}}{d_{n-2}}$, where
\begin{align*}
d_1&=\mathrm{gcd}(e_1,\ldots,e_n),\\
d_2&=\mathrm{gcd}(e_1e_2,e_1e_3,\ldots,e_ie_j,\dots,e_{n-1}e_n),\\
&\vdots\\
d_{n-1}&=\mathrm{gcd}(e_1e_2\cdots e_{n-1},\ldots,e_2\cdots e_{n-1}e_n).
\end{align*}
Moreover, if $\gamma$ is a generator of the~free part of $\mathrm{Pic}^G(C)$ in this case, then we have
$$
4=\mathrm{deg}(K_C)=\mathrm{lcm}(e_1,\ldots, e_n)\left(n-2-\sum_{i=1}^n\frac{1}{e_i}\right)\mathrm{deg}(\gamma).
$$
\end{remark}

Let us show how to compute $\mathrm{Pic}^G(C)$ in some cases.

\begin{example}
\label{example:S4}
Suppose that $\mathrm{Aut}(C)$ contains a~subgroup $G\simeq\mathfrak{S}_4$.
Then $C$ is given in $\mathbb{P}^2_{x,y,z}$ by
$$
x^4+y^4+z^4+\lambda (x^2y^2+x^2z^2+y^2z^2)=0
$$
for some $\lambda\in\mathbb{C}$ such that $\lambda\not\in\{-1,2,-2\}$. One can show that
$$
\mathrm{Aut}(C)\simeq\left\{\aligned
&\mathfrak{S}_4\ \text{if $\lambda\ne 0$ and $\lambda\ne \frac{-3\pm 3\sqrt{7}i}{2}$},\\
&\mumu_4^2\rtimes\mathfrak{S}_3\ \text{if $\lambda=0$},\\
&\mathrm{PSL}_2(\mathbb{F}_7)\ \text{if $\lambda=\frac{-3\pm 3\sqrt{7}i}{2}$}.
\endaligned
\right.
$$
We have $C/G\simeq\mathbb{P}^1$, and it follows from \cite{lmfdb} that the~signature is $[0;2,2,2,3]$.
Thus, using Remark~\ref{remark:Dolgachev}, we see that $\mathrm{Pic}(G,C)\simeq\mathbb{Z}\times\mumu_2^2$,
and the~free part of the~group $\mathrm{Pic}(G,C)$ is generated by $K_{C}$.
Moreover, using GAP, we compute $\mathrm{Hom}(G,\mathbb{C}^*)\simeq H^2(G,\mathbb{C}^*)\simeq\mumu_2$.
Therefore, using Proposition~\ref{proposition:Dolgachev}, we get the~following exact sequence of group homomorphisms:
$$
0\rightarrow\mathbb{Z}\times\mumu_2\rightarrow\mathrm{Pic}^G(C)\rightarrow \mumu_2\rightarrow 0.
$$
We also know from \cite{Disney-Hogg} that $\mathrm{Pic}(C)$ contains two $G$-invariant even theta-characteristics $\theta_1$ and $\theta_2$.
This immediately implies that $\mathrm{Pic}^G(C)=\langle\theta_1,\theta_2\rangle\simeq\mathbb{Z}\times\mumu_2$.
\end{example}

\begin{example}[{\cite{Dolgachev1999}}]
\label{example:PSL-2-7}
Suppose that $\mathrm{Aut}(C)\simeq\mathrm{PSL}_2(\mathbb{F}_7)$.
Then $C$ is given in $\mathbb{P}^2_{x,y,z}$ by
$$
xy^3+yz^3+zx^3=0.
$$
Set~$G=\mathrm{Aut}(C)$. Using Example~\ref{example:Edge}, we conclude that $\mathrm{Pic}^G(C)$ contains an~even theta-characteristic~$\theta$.
Now, arguing as in Example~\ref{example:S4}, we get $\mathrm{Pic}^G(C)=\langle\theta\rangle\simeq\mathbb{Z}$.
\end{example}

\begin{example}
\label{example:Fermat-48-3}
Suppose that $\mathrm{Aut}(C)\simeq\mumu_4^2\rtimes\mathfrak{S}_3$.
Then $C$ is given in $\mathbb{P}^2_{x,y,z}$ by
$$
x^4+z^4+z^4=0,
$$
the group $\mathrm{Aut}(C)$ contains a~unique subgroup isomorphic to $\mumu_4^2\rtimes\mumu_3$,
and $C$ is the~unique plane quartic curve admiting a~faithful $\mumu_4^2\rtimes\mumu_3$-action.
Let $G$ be this subgroup. Then the~signature~is~$[0;3,3,4]$.
Therefore, using Remark~\ref{remark:Dolgachev},
we get $\mathrm{Pic}(G,C)\simeq\mathbb{Z}\times\mumu_3$, where the~free part is generated~by~$K_{C}$.
Since $\mathrm{Hom}(G,\mathbb{C}^*)\simeq\mumu_3$ and $H^2(G,\mathbb{C}^*)\simeq\mumu_4$,
it follows from Proposition~\ref{proposition:Dolgachev} that
$$
\mathrm{Pic}^G(C)/\langle K_C\rangle\simeq\mumu_4.
$$
Moreover, we know from Section~\ref{subsection:48-3} that $\mathrm{Pic}^G(C)$ contains a~divisor $D$ of degree $2$.
Thus, we conclude that $\mathrm{Pic}^G(C)=\langle K_C,D\rangle\simeq\mathbb{Z}\times\mumu_2$, and $K_C-2D$ is a~two-torsion divisor.
\end{example}

\begin{example}
\label{example:Fermat-96-64}
Let $C$ be the~Fermat quartic curve from Example~\ref{example:Fermat-48-3}, and let $G=\mathrm{Aut}(C)\simeq\mumu_4^2\rtimes\mathfrak{S}_3$.
Then the~signature is $[0;2,3,8]$, so it follows from Remark~\ref{remark:Dolgachev} that
$$
\mathrm{Pic}(G,C)\simeq\mathbb{Z}\times\mumu_2,
$$
where the~free part is generated by $K_{C}$.
On can check that $\mathrm{Hom}(G,\mathbb{C}^*)\simeq\mumu_2$ and $H^2(G,\mathbb{C}^*)\simeq\mumu_2$.
We claim that $\mathrm{Pic}^G(C)$ contains no divisors of degree~$2$.
Indeed, if $\mathrm{Pic}^G(C)$ has a~divisor $D$ of degree~$2$, then~$|K_C+D|$ gives a~$G$-equivariant embedding $\phi\colon C\hookrightarrow\mathbb{P}^3$,
which contradicts to Lemmas~\ref{lemma:48-3-unique}~and~\ref{lemma:48-normalizer-PGL},
because $\mumu_2^3.\mathfrak{S}_4$ does not contain subgroups isomorphic to $G$.
Therefore, arguing as in Example~\ref{example:Fermat-48-3}, we see that $\mathrm{Pic}^G(C)=\langle K_C,\delta\rangle\simeq\mathbb{Z}\times\mumu_2$,
where $\delta$ is a~two-torsion divisor.
\end{example}

Using results described in Examples~\ref{example:S4}, \ref{example:PSL-2-7}, \ref{example:Fermat-48-3}, \ref{example:Fermat-96-64}, we get the~ following corollaries:

\begin{corollary}
\label{corollary:S4-nice}
If $\mathrm{Aut}(\mathbb{P}^3,C_6)$ has a~subgroup isomorphic to $\mathfrak{S}_4$,
then one of the~following holds:
\begin{itemize}
\item $\mathrm{Aut}(\mathbb{P}^3,C_6)\simeq\mathfrak{S}_4$ and $\mathrm{Aut}(X)\simeq\mathfrak{S}_4\times\mumu_2$,
\item $\mathrm{Aut}(\mathbb{P}^3,C_6)\simeq\mathrm{PSL}_2(\mathbb{F}_7)$ and $\mathrm{Aut}(X)\simeq\mathrm{PSL}_2(\mathbb{F}_7)\times\mumu_2$.
\end{itemize}
\end{corollary}

\begin{corollary}
\label{corollary:Klein-nice}
The smooth Fano threefold described in Example~\ref{example:Edge} is the~unique smooth Fano threefold in the~deformation family \textnumero 2.12
that admits a~faithful action of the~group $\mathrm{PSL}_2(\mathbb{F}_7)$.
\end{corollary}

\begin{corollary}
\label{corollary:Fermat-nice}
The smooth Fano threefold described in Section~\ref{subsection:48-3} is the~only smooth Fano~\mbox{threefold} in the~family \textnumero 2.12
that admits a~faithful action of the~group $\mumu_4^2\rtimes\mumu_3$.
\end{corollary}

\begin{proof}
Suppose that $\mathrm{Aut}(X)$ has a subgroup isomorphic to $\mumu_4^2\rtimes\mumu_3$.
Then arguing as in Example~\ref{example:Fermat-96-64},
we see that $\mathrm{Aut}(\mathbb{P}^3,C_6)\simeq\mumu_4^2\rtimes\mumu_3$,
and $\mathrm{Aut}(\mathbb{P}^3,C_6)$ is conjugate to the~subgroup $G$ that has been described in Section~\ref{subsection:48-3}.
Thus, the~required assertion follows from Theorem~\ref{theorem:d-6-g-3}.
\end{proof}

Now, we are ready to prove Theorem~B.

\begin{proof}[Proof of Theorem~B]
It is enough to show that the~automorphism group $\mathrm{Aut}(X)$ is isomorphic to a~subgroup of $\mathrm{PSL}_2(\mathbb{F}_7)\times\mumu_2$ or $\mumu_4^2\rtimes\mathfrak{S}_3$.
Suppose this is not true. Let us seek for a contradiction.

Let $G=\mathrm{Aut}(C,[D])$.
Then $G$ is also not isomorphic to a subgroup of $\mathrm{PSL}_2(\mathbb{F}_7)\times\mumu_2$ or $\mumu_4^2\rtimes\mathfrak{S}_3$.
Therefore, using \cite{Disney-Hogg} and the~classification of automorphism groups of smooth plane quartic curves,
we see that $D$ is not an even theta-characteristic.
So, by Corollary~\ref{corollary:Aut-2}, the~group $\mathrm{Aut}(X)$ is isomorphic to a subgroup of the~group $\mathrm{Aut}(C)$.

Hence, using the~classification of automorphism groups of smooth plane quartic curves again,
we~conclude that the~group $\mathrm{Aut}(X)$ is isomorphic to one of the~following groups:
\begin{center}
$\mumu_9$, $\mumu_{12}$, $\mathrm{SL}_{2}(\mathbb{F}_3)$ (GAP ID is [24,3]), $\mumu_4.\mathfrak{A}_4$ (GAP ID is [48,33]),
\end{center}
and it follows from Corollary~\ref{corollary:Aut} that either $G=\mathrm{Aut}(X)$ or $G$ is a subgroup in $\mathrm{Aut}(X)$ of index~$2$.
Thus, we have the~following possibilities:

\begin{center}
\renewcommand\arraystretch{1.5}
\begin{tabular}{|c|c|c|c|c|c|c|}
\hline
$\mathrm{Aut}(X)$ & $\mumu_9$ & $\mumu_{12}$ & $\mumu_{12}$ & $\mathrm{SL}_{2}(\mathbb{F}_3)$ & $\mumu_4.\mathfrak{A}_4$ & $\mumu_4.\mathfrak{A}_4$ \\
\hline
$G$               & $\mumu_9$ & $\mumu_{6}$ & $\mumu_{12}$ & $\mathrm{SL}_{2}(\mathbb{F}_3)$ & $\mathrm{SL}_{2}(\mathbb{F}_3)$ & $\mumu_4.\mathfrak{A}_4$ \\
\hline
\end{tabular}
\end{center}

Recall that $D$ is a divisor on the~quartic curve $C$ such that $\mathrm{deg}(D)=2$, the~divisor $D$ satisfies~\ref{equation:very-ample}, and its class $[D]\in\mathrm{Pic}(C)$ is $G$-invariant.
Let us show that in each of our cases, such $D$ does not~exist.

First, using \cite{DolgachevBook,Bars}, Proposition~\ref{proposition:Dolgachev} and Remark~\ref{remark:Dolgachev},
we can describe the~equation of the~curve~$C$, the~signature of the~action of the~group $G$ on the~curve~$C$,
the structure of the~group $\mathrm{Pic}^G(S)$,
and the~degree of a generator $\gamma$ of the~free part of the~group $\mathrm{Pic}^G(C)$.
This gives the~following possibilities:

\begin{center}\renewcommand\arraystretch{1.5}
\begin{tabular}{|c|c|c|c|c|}
\hline
$G$ & Equation of $C$ & Signature & Structure of $\mathrm{Pic}^G(S)$ & $\mathrm{deg}(\gamma)$\\
\hline
$\mumu_{6}$&$y^4-x^3z+z^4=0$&$[0;2,3,3,6]$&$\mathbb{Z}\oplus\mathbb{Z}_3$&$1$\\
\hline
$\mumu_{9}$&$y^3z-x(x^3-z^3)=0$&$[0;3,9,9]$&$\mathbb{Z}\oplus\mathbb{Z}_3$&$1$\\
\hline
$\mumu_{12}$&$y^4-x^3z+z^4=0$&$[0;3,4,12]$&$\mathbb{Z}$&$1$\\
\hline
$\mathrm{SL}_2(3)$&$y^4-x^3z+z^4=0$&$[0;2,3,6]$&$\mathbb{Z}\oplus\mathbb{Z}_6$&$4$\\
\hline
$\mumu_4.\mathfrak{A}_4$&$y^4-x^3z+z^4=0$&$[0;2,3,12]$&$\mathbb{Z}$&$4$\\
\hline
\end{tabular}
\end{center}

In particular, if $G\cong\mathrm{SL}_2(3)$ or $G\simeq\mumu_4.\mathfrak{A}_4$, then $C$ does not have $G$-invariant divisors of degree~$2$.
Hence, we see that $G$ is isomorphic to one of the~following groups: $\mumu_6$, $\mumu_9$, $\mumu_{12}$.

Suppose that $G\simeq\mumu_{12}$. Then the~action of $G$ on $C$ is generated by
$$
[x:y:z]\mapsto\big[\omega_3x:iy:z\big],
$$
where $\omega_3$ is a primitive cube root of the~unity.
Then $G$ fixes the~point $P=[1:0:0]$, which implies that $\mathrm{Pic}^G(S)=\mathbb{Z}[P]$,
so that $D\sim 2P$, which contradicts to our assumption that $D$ satisfies \ref{equation:very-ample}.

Assume now that $G\simeq\mumu_9$. Then the~$G$-action on the~curve is given by
$$
[x:y:z]\mapsto\big[\omega_9x:\omega_9^{-3}y:z\big],
$$
where $\omega_9$ is a primitive ninth root of the~unity. Set $P_1=[1:0:0]$ and $P_2=[0:1:0]$.
Then
$$
\mathrm{Pic}^G(S)=\langle P_1,P_2\rangle,
$$
because $P_1$ and $P_2$ are fixed by the action of the group $G$, and the divisor $P_1-P_2$ is a $3$-torsion. Then $D$ is linearly equivalent to $2P_1$, $2P_2$ or $P_1+P_2$, which contradicts \ref{equation:very-ample}.

Finally, consider the~case where $G$ is isomorphic to $\mumu_6$. Then the~$G$-action is given by
 $$
[x:y:z]\mapsto\big[x:-y:\omega_3z\big],
$$
Set $P=[0:0:1]$, $\Sigma_2=[1:i:0]+[1:-i:0]$, and $\Sigma_2^\prime=[1:1:0]+[1:-1:0]$.
Then
$$
K_C\sim 4P\sim\Sigma_2+\Sigma_2^\prime,
$$
and the divisors $P$, $\Sigma_2$, $\Sigma^\prime$ are $G$-invariant. This gives $\mathrm{Pic}^G(S)=\langle P,\Sigma_2\rangle$, and  $2P-\Sigma_2$ is a $3$-torsion.
Then $D$ is linearly equivalent to $2P$, $\Sigma_2$, $\Sigma_2^\prime$, which contradicts \ref{equation:very-ample}.
\end{proof}

\end{document}